\documentclass[10pt]{amsart}

\usepackage[applemac]{inputenc}
\usepackage{color}
\pdfoutput=1
\usepackage{geometry} % see geometry.pdf on how to lay out the page. Shere's lots.
\usepackage{graphicx,url}
\usepackage{amssymb,amsmath, amsthm}
\usepackage{epstopdf}
\usepackage{subfig,lscape}
\usepackage[applemac]{inputenc} 

\newcommand{\re}{\mathbb{R}}

\newcommand{\cor}{\mathcal}
\usepackage{graphicx,url}

\newtheorem{theorem}{Theorem}
\newtheorem{proposition}{Proposition}
\newtheorem{lemma}{Lemma}
\newtheorem{remark}{Remark}[section]

\title{Empirical $L^2$-distance test statistics for ergodic diffusions}
\numberwithin{equation}{section}
\newcommand{\de}{\mathrm{d}}
\newcommand{\me}{\mathrm{m}}
\newcommand{\ee}{\mathrm{e}}
\newcommand{\oo}{\mathbf{o}}

\allowdisplaybreaks
\date{\today}

\author{A. De Gregorio}
\author{S.M. Iacus}
\address{Department of Statistical Sciences, ``Sapienza" University of Rome,
P.le Aldo Moro, 5 - 00185, Rome, Italy}
\address{Department of Economics, Management and Quantitative Methods, University of Milan,
 Via Conservatorio 7, 20122 - Milan, Italy }
\email{alessandro.degregorio@uniroma1.it}
\email{stefano.iacus@unimi.it}

\begin{document}

\maketitle

\begin{abstract} The aim of this paper is to introduce a new type of test statistic for simple null hypothesis on one-dimensional ergodic diffusion processes sampled at discrete times. We deal with a quasi-likelihood approach for stochastic differential equations (i.e. local gaussian approximation of the transition functions) and define a test statistic by means of the empirical $L^2$-distance between quasi-likelihoods. We prove that the introduced test statistic is asymptotically distribution free; namely it weakly converges to a $\chi^2$ random variable. Furthermore, we study the power under local alternatives of the parametric test. We show by the Monte Carlo analysis that, in the small sample case, the introduced test seems to perform better than other tests proposed in literature. 
\end{abstract}


\section{Introduction}
 Let  $(\Omega,\cor F,{\bf F}=(\cor F_t)_{t\geq 0},P)$ be a filtered complete probability space.
Let us consider a $1$-dimensional processes $X=(X_t)_{t\geq 0}$ solution to  the following stochastic differential equation
\begin{equation}\label{sde}
\de X_t=b(\alpha,X_{t})\de t+\sigma(\beta,X_{t})\de W_t,\quad X_0=x_0,
\end{equation}
where $x_0$ is a deterministic initial value. We assume that $b: \Theta_{\alpha}\times\mathbb{R}\to \mathbb{R}$, $\sigma:\Theta_\beta\times\mathbb{R} \to \mathbb{R}$ are Borel known functions (up to $\alpha$ and $\beta$) and  $(W_t)_{t\geq 0}$ is a one-dimensional standard $\cor F_t$-Brownian motion. Furthermore, $\alpha\in\Theta_\alpha\subset \re^{m_1},\beta\in\Theta_\beta\subset \re^{m_2},m_1, m_2\in\mathbb N,$ are unknown parameters and $\theta=(\alpha,\beta)\in\Theta:=\Theta_\alpha\times\Theta_\beta,$ where $\Theta$ represents a compact subset of $\re^{m_1+m_2}$. We denote by $\theta_0:=(\alpha_0,\beta_0)$ the true value of $\theta$ and assume that $\theta_0\in$ Int$(\Theta).$ 
%The data are represented by the discrete observations of the trajectory of $X;$ that is $(X_{t_i^n})_{0\leq i\leq n}$ with $t_i^n:=i\Delta_n$ and $\Delta_n$ is called discretization step.

The sample path of $X$ is observed only at $n+1$ equidistant discrete times $t_i^n$, such that $t_i^n-t_{i-1}^n=\Delta_n<\infty$ for $i=1,...,n,$ (with $t_0^n=0$). Therefore the data, denoted by $(X_{t_i^n})_{0\leq i\leq n},$ are the discrete observations of the sample path of $X.$ Let $p$ be an integer with $p\geq 2.$ The asymptotic scheme adopted in this paper is the following: $T=n\Delta_n\to \infty$, $\Delta_n\to 0$ and $n\Delta_n^p\to 0$ as $n\to \infty$. This scheme is called rapidly increasing design, i.e. the number of observations grows over time but no so fast.

This setting is useful, for instance, in the analysis of financial time series. In mathematical finance and econometric theory, diffusion processes described by the stochastic differential equations \eqref{sde} play a central role. Indeed, they have been used to model the behavior of stock prices, exchange rates and interest rates. The underlying stochastic evolution of the financial assets can be thought continuous in time, although the data are always recorded at discrete instants (e.g. weekly, daily or each minute). For these reasons, the estimation problems for discretely observed stochastic differential equations have been tackled by many authors with different approaches (see, for instance,  \cite{fz}, \cite{Yos92}, \cite{gcj},  \cite{bibbsor}, \cite{kess}, \cite{kesssor}, \cite{ait}, \cite{gobet}, \cite{jacod}, \cite{ait2}, \cite{DGIcp}, \cite{phillyu}, \cite{Yos11}, \cite{uchyos12}, \cite{li}, \cite{uchyos14}, \cite{kama}). For clustering time series arising from discrete observations of diffusion processes  \cite{DGIclust} propose a new dissimilarity measure based on the $L^1$ distance between the Markov operators. The change-point problem in the diffusion term of a stochastic differential equation has been considered in \cite{DGIcp} and \cite{iacusyos}. In  \cite{iacusyosuc}, the authors faced the estimation problem for hidden diffusion processes observed at discrete times. An adaptive Lasso-type estimator is proposed in \cite{DGIlasso}. For the simulation and the practical implementation of the statistical inference for stochastic differential equations see \cite{iacus}, \cite{iacus2} and \cite{iacusyosyui}. 

We also recall that the statistical inference for continuously observed ergodic diffusions is a well-developed research topic; on this point the reader can consult \cite{kut}.

The main object of interest of the present paper is the problem of testing parametric hypotheses for diffusion processes from discrete observations. This research topic is less developed in literature. %Tests for model specification of a parametric diffusion process have been proposed, for instance, by \cite{aittest} and \cite{chen}. 
It is well-known that for testing two simple alternative hypotheses, the Neyman-Pearson lemma provides a procedure based on the likelihood ratio which leads to the uniformly most powerful test. In the other cases uniformly most powerful tests do not exist and for this reason the research of new criteria is justified. 

For discretely observed stochastic differential equations,  \cite{kituch} introduced and studied the asymptotic behavior of three kinds of test statistics: likelihood ratio type test statistic, Wald type test statistic and Rao's score type test statistic. 

Another possible approach is based on the divergences. Indeed, several statistical divergence measures (which are not necessarily a metric) and distances have been introduced in order to decide if two probability distributions are close or far. The main goal of this metric is to make ``easy to distinguish'' between a pair of distributions which are far from each other than between those which are closer. These tools have been used for testing hypotheses in parametric models. The reader can consult on this point, for example, \cite{mpv} and \cite{pardo}. For stochastic differential equations sampled at discrete times, \cite{DGItest} introduced a family of test statistics (for $p=2$ and $n\Delta_n^2\to 0$) based on empirical $\phi$-divergences.

%Another way to quantify the ``distance'' between two probability measures is to consider the divergence measure which are not n. For example we provide the definition of the $\phi$-divergence.

We consider the following hypotheses testing problem concerning the vector parameter $\theta$  
$$H_0:\theta=\theta_0,\quad\text{vs}\quad  H_1:\theta\neq \theta_0,$$
and assume that $X$ is observed at discrete times; that is the data  $(X_{t_i^n})_{0\leq i\leq n}$ are available. In this work we study different test statistics with respect to those used in  \cite{DGItest} and \cite{kituch}. Indeed, the purpose of this paper is to propose a methodology based on a suitable ``distance'' between the approximated transition functions. This idea follows from the observation that in the case of continuous observations of \eqref{sde}, we could define the $L^2$-distance between the continuous loglikelihood. Clearly this approach is not useful in our framework and then, similarly to the aforementioned papers, we consider the local gaussian approximation of the transition density of the process $X$ from $X_{t_{i-1}}$ to $X_{t_i}.$ In other words, we resort the quasi-likelihood function introduced in \cite{kess}, defined by means of an approximation with higher order correction terms to relax the condition of convergence of $\Delta_n$ to zero.  Therefore, let $\texttt{l}_{p,i}(\theta),\theta\in\Theta,$ be the approximated log-transition function  from $X_{t_{i-1}}$ to $X_{t_i}$ representing the parametric model \eqref{sde}. We deal with \begin{equation*}
\mathbb{D}_{p,n}(\theta_1,\theta_2):=\frac1n\sum_{i=1}^n[\texttt{l}_{p,i}(\theta_1)-\texttt{l}_{p,i}(\theta_2)]^2,\quad \theta_1,\theta_2\in\Theta,
\end{equation*}
which can be interpreted as the empirical $L^2$-distance between two loglikelihoods. 
If $\hat \theta_{p,n}$ is the maximum quasi-likelihood estimator introduced in \cite{kess}, we are able to prove that, under $H_0,$ the test statistic $$T_{p,n}(\hat \theta_{p,n},\theta_0):=n\mathbb{D}_{p,n}(\hat \theta_{p,n},\theta_0)$$ is asymptotically distribution free; i.e. it converges in distribution to a chi squared random variable. Furthermore, we study the power function of the test under local alternatives.

The paper is organized as follows. Section \ref{not} contains the notations and the assumptions of the paper. The contrast function arising from the quasi-likelihood approach is briefly discussed in Section \ref{quasilik}. In the same section we define the maximum quasi-likelihood estimator and recall its main asymptotic properties. In Section \ref{test} we introduce and study a test statistic for the hypotheses problem $H_0:\theta=\theta_0$ vs $H_1: \theta\neq \theta_0$.
The proposed new test statistic shares the same asymptotic properties of the other test statistics presented in the literature. Therefore, to justify its use in practice among its competitors,  a numerical study is included in Section \ref{sec:numerics} which contains a comparison of several test statistics in the ``small sample'' case, i.e., when the asymptotic conditions are not met. Our numerical analysis shows that, at least for $p=2,$ the performance of $T_{2,n}$ is very good. The proofs are collected in Section \ref{proofs}.

It is worth to point out that for the sake of simplicity in this paper a 1-dimensional diffusion is treated. Nevertheless, it is possible to extend our methodology to the multidimensional stochastic differential equations setting. 

\section{Notations and assumptions}\label{not}

Throughout this paper, we will use the following notation.
\begin{itemize}
\item $\theta:=(\alpha,\beta)$ and $\alpha_0,\beta_0$ and $\theta_0$ denote the true values of $\alpha,\beta$ and $\theta$ respectively. 
\item $c(\beta,x)=\sigma^2(\beta,x).$
\item $C$ is a positive constant. If $C$ depends on a fixed quantity, for instance an integer $k,$ we may write $C_k.$  
\item $\partial_{\alpha_h}:=\frac{\partial}{\partial \alpha_h},\partial_{\beta_k}:=\frac{\partial}{\partial \beta_k},%\partial^2_{\alpha_h}:=\frac{\partial^2}{\partial \alpha_h^2},\partial^2_{\beta_k}:=\frac{\partial^2}{\partial \beta_k^2}, 
\partial^2_{\alpha_h\alpha_k}:=\frac{\partial^2}{\partial\alpha_h\partial\alpha_k}, h,k=1,..., m_1, \partial^2_{\beta_h\beta_k}:=\frac{\partial^2}{\partial\beta_h\partial\beta_k}, h,k=1,...,m_2,\partial^2_{\alpha_h\beta_k}:=\frac{\partial^2}{\partial\alpha_h\partial\beta_k}, h=1,...,m_1, k=1,...,m_2,$ $\partial_\theta:=(\partial_\alpha,\partial_\beta)',$ where $\partial_\alpha:=(\partial_{\alpha_1},...,\partial_{\alpha_{m_1}})'$ and $\partial_\beta:=(\partial_{\beta_1},...,\partial_{\beta_{m_2}})',$ 
$\partial_\theta^2:=[\partial_{\alpha_j\beta_k}^2]_{h=1,...,m_1, k=1,...,m_2}.$
\item If $f:\Theta\times \mathbb R\to \mathbb R,$ we denote by $f_{i-1}(\theta)$ the value $f(\theta, X_{t_{i-1}^n})$; for instance $c(\beta, X_{t_{i-1}^n})=c_{i-1}(\beta)$.
\item For $0\leq i\leq n, t_i^n:=i\Delta_n$  and $\mathcal G_i^n:=\sigma(W_s,s\leq t_i^n).$
\item The random sample is given by ${\bf X}_n:=(X_{t_i^n})_{0\leq i\leq n}$ and $X_i:=X_{t_i^n}.$
\item The probability law of \eqref{sde} is denoted by $P_\theta$ and $E_\theta^{i-1}[\cdot]:=E_\theta[\cdot|\mathcal G_{i-1}^n].$ We set $P_0:=P_{\theta_0}$ and $E_0^{i-1}[\cdot]:= E_{\theta_0}^{i-1}[\cdot].$
\item $\overset{P_\theta}{\underset{n\to\infty}{\longrightarrow}} $ and $\overset{d}{\underset{n\to\infty}{\longrightarrow}} $ stand for the convergence in probability and in distribution, respectively.

\item Let $F_n:\Theta\times \mathbb R^{n}\to \mathbb R$ and $F:\Theta\to \mathbb R;$ $``F_n(\theta, {\bf X}_n)\overset{P_\theta}{\underset{n\to\infty}{\longrightarrow}} F(\theta)$ uniformly in $\theta"$ stands for
$$\sup_{\theta\in\Theta}\left|F_n(\theta, {\bf X}_n)-F(\theta)\right|\overset{P_\theta}{\underset{n\to\infty}{\longrightarrow}} 0.$$
Furthermore, if $F_n(\theta, {\bf X}_n)\overset{P_\theta}{\underset{n\to\infty}{\longrightarrow}} 0$ uniformly in $\theta$ we set $$F_n(\theta, {\bf X}_n)=\oo_{P_\theta}(1).$$

\item Let $u_n$ be a $\mathbb{R}$-valued sequence. We indicate by $R$ a function $\Theta\times\mathbb{R}^2\to \mathbb{R}$ for which there exists a constant $C$ such that
$$R(\theta,u_n,x)\leq u_nC(1+|x|)^C,\quad \text{for all}\,\theta\in\Theta, x\in \mathbb{R}^2, n\in\mathbb{N}.$$
Let us set $R_{i-1}(\Delta_n^k):=R(\theta,\Delta_n^k,X_{i-1}).$
\item For a $m\times n$ matrix $A$, $||A||^2=\text{tr}(AA')=\sum_{i=1}^m\sum_{j=1}^n |A_{ij}|^2.$
\end{itemize}

%Let $C_{\uparrow}^{k}(\mathbb R, \mathbb R)$ be the space of all functions $f$ such that $f(x)\in C^k(\mathbb R,\mathbb R)$ and $f$ and its derivatives up to order $k$ are of polynomial growth.

 Let $C_{\uparrow}^{k,h}(\mathbb R\times \Theta; \mathbb R)$ be the space of all functions $f$ such that:
\begin{itemize}
\item[(i)] $f(\theta,x)$ is a $\mathbb R$-valued function on $ \Theta\times\mathbb R;$
\item[(ii)] $f(\theta,x)$ is continuously differentiable with respect to $x$ up to order $k\geq 1$ for all $\theta;$ these $x$-derivatives up to order $k$ are of polynomial growth in $x,$ uniformly in $\theta$;% that is $\sup_{\theta\in\Theta}|\partial_x^k f(x,\theta)|\leq C(1+|x|)^C; $
\item[(iii)] $f(\theta,x)$ and all $x$-derivatives up to order $k\geq 1,$ are $h\geq 1$ times continuously differentiable with respect to $\theta$ for all $x\in\mathbb R.$ Moreover, these derivatives up to the $h$-th order with respect to $\theta$ are of polynomial growth in $x,$ uniformly in $\theta$.% that is $\sup_{\theta\in\Theta}|\partial_\theta^l\partial_x^k f(x,\theta)|\leq C(1+|x|)^C. $
\end{itemize}

We need  some standard assumptions on the regularity of the process $X.$
\begin{itemize}
\item[$ A_1.$]  (Existence and Uniqueness) There exists a constant $C$ such that
$$\sup_{\alpha\in\Theta_\alpha}|b(\alpha,x)-b(\alpha,y)|+\sup_{\beta\in\Theta_\beta}|\sigma(\beta,x)-\sigma(\beta,y)|\leq C|x-y|.$$

\item[$ A_2.$] (Ergodicity) The process $X$ is ergodic for $\theta=\theta_0$ with
invariant probability measure $\pi_0(\de x)$. Thus
$$\frac1T\int_0^Tf(X_t)\de t\overset{P_\theta}{\underset{T\to\infty}{\longrightarrow}} \int f(x)\pi_0(\de x),$$
 where $f\in L^1(\pi_0)$. Furthermore, we assume that $\pi_0$ admits all moments finite.

\item[$ A_3.$] $\inf_{x,\beta}\sigma(\beta,x)>0.$

\item[$ A_4.$]  (Moments) For all $q\geq 0$ and for all $\theta\in\Theta$, $\sup_t
E|X_t|^q<\infty$.
\item[$A_5.$] $[k]$ (Smoothness) $b\in C_{\uparrow}^{k,3}(\Theta_\alpha\times \mathbb R,\mathbb R)$ and $\sigma\in C_{\uparrow}^{k,3}(\Theta_\beta\times \mathbb R,\mathbb R).$
\item[$ A_6.$] (Identifiability) If the coefficients $b(\alpha,x)=b(\alpha_0,x)$ and
$\sigma(\beta,x)=\sigma(\beta_0,x)$  for all
$x$  ($\pi_{0}$-almost surely), then $\alpha=\alpha_0$ and $\beta=\beta_0$.
\end{itemize}

Let $L_\theta$ the infinitesimal generator of $X$ with domain given by $C^2(\mathbb R)$ (the space of the twice continuously differentiable function on $\mathbb R$); that is if $f\in C^2(\mathbb R)$
$$L_\theta f(x):=b(\alpha,x)\frac{\partial f}{\partial x}(x)+\frac{c(\beta,x)}{2}\frac{\partial^2 f}{\partial x^2}(x),\quad L_0:=L_{\theta_0}.$$
Under the assumption $A_5$$[2(j-1)]$ we can define $L_\theta^j:=L_\theta \circ L_\theta^{j-1}$ with domain $C^{2j}(\mathbb R)$ and $L_\theta^0=$Id.

We conclude this section with some well-known examples of ergodic diffusion processes belonging to the class \eqref{sde}:

\begin{itemize}
\item the Ornstein-Uhlenbeck or Vasicek model is the unique solution to 

\begin{equation}\label{ou}
\de X_t=\alpha_1(\alpha_2-X_t)\de t+\beta_1 \de W_t,\quad X_0=x_0,
\end{equation}
where $b(\alpha_1,\alpha_2,x)=\alpha_1(\alpha_2-x)$ and $\sigma(\beta_1,x)=\beta_1$ with $\alpha_1,\alpha_2\in \mathbb R$ and $\beta_1>0.$ This stochastic process is a Gaussian process and it is often used in finance where $\beta_1$ is the volatility, $\alpha_2$ is the long-run equilibrium of the model and $\alpha_1$ is the speed of mean reversion. For $\alpha_1>0$ the Vasicek process is ergodic with invariant law $\pi_0$ given by a Gaussian law with mean $\alpha_2$ and variance $\frac{\beta_1^2}{2\alpha_1}.$ It is easy to check that all the conditions $A_1-A_6$ fulfill;

\item the Cox-Ingersoll-Ross (CIR) process is the solution to

\begin{equation}\label{cir}
\de X_t=\alpha_1(\alpha_2-X_t)\de t+\beta_1 \sqrt{X_t}\de W_t,\quad X_0=x_0>0,
\end{equation}
where  $b(\alpha_1,\alpha_2,x)=\alpha_1( \alpha_2-x)$ and $\sigma(\beta_1,x)=\beta_1\sqrt x$ with $\alpha_1,\alpha_2,\beta_1>0.$ If $2\alpha_1\alpha_2>\beta_1^2$ the process is strictly positive, otherwise non negative. This model has a conditional density given by the non central $\chi^2$ distribution. The CIR process is useful in the description of short-term interest rates and admits invariant law $\pi_0$ given by a Gamma distribution with shape parameter $\frac{2\alpha_1\alpha_2}{\beta_1^2}$ and scale parameter $\frac{\beta_1^2}{2\alpha_1}.$ If \eqref{cir} is strictly positive, we can prove that the above assumptions hold true.
\end{itemize}

\section{Preliminaries on the quasi-likelihood function}\label{quasilik}

We briefly recall the quasi-likelihood function introduced by \cite{kess} based on the It\^o-Taylor expansion. The main problem in the statistical analysis of the diffusion process $X$ is that its transition density is in general unknown and then the likelihood function is unknown as well. To overcome this difficulty one can discretizes the sample path of $X$ by means of Euler-Maruyama's scheme; namely
\begin{align}\label{eq:emscheme}
X_{i}-X_{i-1}=\int_{t_{i-1}^n}^{t_{i}^n}b(\alpha,X_s)\de s+\int_{t_{i-1}^n}^{t_{i}^n}\sigma(\beta,X_s)\de W_s\simeq b_{i-1}(\alpha)\Delta_n+\sigma_{i-1}(\beta)(W_{t_i^n}-W_{t_{i-1}^n}).
\end{align}
Hence \eqref{eq:emscheme} leads to consider a local-Gaussian approximation to the transition density; that is 
$$\mathcal L(X_i|X_{i-1})\simeq N(b_{i-1}(\alpha)\Delta_n,c_{i-1}(\beta)\Delta_n )$$
and the approximated loglikelihood function of the random sample ${\bf X}_n,$ called quasi-loglikelihood function, becomes
\begin{equation}\label{locga}
l_n(\theta):=\frac12\sum_{i=1}^n\left\{\frac{(X_i-X_{i-1}-b_{i-1}(\alpha)\Delta_n)^2}{c_{i-1}(\beta)\Delta_n}+\log c_{i-1}(\beta) \right\}.
\end{equation}
This approach suggests to consider the mean and the variance of the transition density of $X;$ that is
\begin{equation}
\me(\theta,X_{i-1}):=E_{\theta}[X_{i}|X_{i-1}],\quad \me_2(\theta,X_{i-1}):=E_{\theta}[(X_{i}-\me(\theta,X_{i-1}))^2|X_{i-1}],
\end{equation}
and assume
$$\mathcal L(X_i|X_{i-1})\simeq N(\me(\theta,X_{i-1}),\me_2(\theta,X_{i-1})).$$
Thus we can consider as contrast function the following one
\begin{equation}\label{eq:qcontrast}
\frac12\sum_{i=1}^n\left\{\frac{(X_i-\me(\theta,X_{i-1}))^2}{\me_2(\theta,X_{i-1})}+\log\me_2(\theta,X_{i-1})\right\}.
\end{equation}
Nevertheless, \eqref{eq:qcontrast} does not have a closed form because $\me(\theta,X_{i-1})$ and $\me_2(\theta,X_{i-1})$ are unknown. Therefore we substitute in \eqref{eq:qcontrast} closed approximations of $\me$ and $\me_2$ based on the It\^o-Taylor expansion. 

Let $f(y):=y,$
for $l\geq 0,$ under the assumption $A_5[2l]$, we have the following approximation (see Lemma 1, \cite{kess})
\begin{equation}\label{eq:appmean}
\me(\theta,X_{i-1})=r_l(\Delta_n,X_{i-1},\theta)+R(\theta,\Delta_n^{l+1},X_{i-1})
\end{equation}
where
$$r_l(\Delta_n,X_{i-1},\theta):=\sum_{i=0}^l \frac{\Delta_n^i}{i!}L_\theta^i f(x).$$
%\end{lemma}
Now let us consider the function $(y-r_l(\Delta_n,X_{i-1},\theta))^2,$ which is for fixed $x,y$ and $\theta$ a polynomial in $\Delta_n$ of degree $2l.$ We indicate by $\overline g_{\Delta_n,x,\theta,l}(y)$ the sum of its first terms up to degree $l;$ that is $\overline g_{\Delta_n,x,\theta,l}(y)=\sum_{j=0}^l \Delta_n^j \overline g_{x,\theta}^j(y)$ where
\begin{align}
& \overline g_{x,\theta}^0(y)=(y-x)^2\label{eq:g0}\\
&\overline g_{x,\theta}^1(y)=-2(y-x)L_\theta f(x)\label{eq:g2}\\
 &\overline g_{x,\theta}^j(y)=-2(y-x)\frac{L_\theta^j f(x)}{j!}+\sum_{r,s\geq 1,r+s=j}\frac{L_\theta^r f(x)}{r!}\frac{L_\theta^s f(x)}{s!},\quad 2\leq j\leq l.\label{eq:gj}
\end{align}
Under the assumption $A_5[2(l-1)]$(i), we have that $L_\theta^r\overline g_{x,\theta}^j(y)$ is well-defined for $r+j=l$ and we set
\begin{equation}\label{eq:gamma}\Gamma_l(\Delta_n,x,\theta):=\sum_{j=0}^l\Delta_n^j\sum_{r=0}^{l-j}\frac{\Delta_n^r}{r!}L_\theta^r\overline g_{x,\theta}^j(x):=\sum_{j=0}^l\Delta_n^j\gamma_j(\theta,x),\end{equation} where $\gamma_j(\theta,x)$ are the coefficients of $\Delta_n^j$. Therefore by \eqref{eq:g0} to \eqref{eq:gamma}, we obtain, for instance, 
\begin{align*}
\gamma_0(\theta,x)&=L_\theta^0  \overline g_{x,\theta}^0(x)=0\\
\gamma_1(\theta,x)&=L_\theta  \overline g_{x,\theta}^0(x)=c(\beta,x)\\
\gamma_2(\theta,x)&=\frac{L_\theta^2  \overline g_{x,\theta}^0}{2}(x)+L_\theta  \overline g_{x,\theta}^1(x)+L_\theta^0  \overline g_{x,\theta}^2(x)\\
&=\frac12\left[b(\alpha,x)\frac{\partial }{\partial y}c(\beta,x)+2c(\beta,x)\frac{\partial }{\partial y}b(\alpha,x)\right]+\frac{c(\beta,x)}{4}\frac{\partial^2 }{\partial y^2} c(\beta,x)
\end{align*}

Let
$$\Gamma_l(\Delta_n,x,\theta):=\Delta_n c(\beta,x)[1+\overline \Gamma_l(\Delta_n,x,\theta)]$$
where
$\overline \Gamma_l(\Delta_n,x,\theta):=\frac{\sum_{j=2}^l\Delta_n^j\gamma_j(\theta,x)}{\Delta_n c(\beta,x)}.$
For $l\geq 0,$ under the assumption $A_5[2l]$(i), we have that (see Lemma 2, \cite{kess})
\begin{equation}\label{eq:appvar}
\me_2(\theta,X_{i-1})=\Delta_n c_{i-1}(\beta)[1+\overline \Gamma_l(\Delta_n,X_{i-1},\theta)]+R(\theta,\Delta_n^{l+1},X_{i-1}).
%E_\theta[(X_i-r_l(\Delta_n,X_{i-1},\theta))^2|X_{i-1}]
\end{equation}

It seems quite natural at this point to substitute  \eqref{eq:appmean} and \eqref{eq:appvar} into the expression \eqref{eq:qcontrast}. Nevertheless, in order to avoid technical difficulties related to the control of denominator and logarithmic we consider a further expansion in $\Delta_n$ of $(1+\overline\Gamma_l)^{-1}$ and $\log(1+\overline\Gamma_l)$.

Let $k_0=[p/2].$ Under the assumption $A_5[2k_0]$(i), we define the quasi-loglikelihood function of ${\bf X}_n$ as
\begin{align}\label{eq:contrast}
%\mathbb H_{p,n}(\theta,{\bf X}_n):=\sum_{i=1}^n\mathbb H_{p,i}(\theta,X_{i-1},X_i)
l_{p,n}(\theta):=l_{p,n}(\theta,{\bf X}_n):=\sum_{i=1}^n\texttt{l}_{p,i}(\theta)
\end{align}
where
\begin{align}\label{eq:contrast2}
\texttt{l}_{p,i}(\theta)&:=\frac{(X_i-r_{k_0}(\Delta_n,X_{i-1},\theta))^2}{2\Delta_nc_{i-1}(\beta)}\left\{1+\sum_{j=1}^{k_0}\Delta_n^j \de_j(\theta,X_{i-1})\right\}\\
&\quad+\frac12\left\{\log c_{i-1}(\beta)+\sum_{j=1}^{k_0}\Delta_n^j \ee_j(\theta,X_{i-1})\right\}\notag
\end{align}
and $\de_j,$ resp. $\ee_j,$ is the coefficient of $\Delta_n^j$ in the Taylor expansion of $(1+\overline\Gamma_{k_0+1}(\Delta_n,x,\theta))^{-1},$ resp. $\log(1+\overline\Gamma_{k_0+1}(\Delta_n,x,\theta)).$ It is not hard to show that, for example,
\begin{align*}
&\de_1(\theta,x)=-\ee_1(\theta,x)=-\frac{\gamma_2(\theta,x)}{c(\beta,x)},\\
&\de_2(\theta,x)=-\ee_2(\theta,x)=\frac{1}{c(\beta,x)}\left[\frac{\gamma_2^2(\theta,x)}{c(\beta,x)}-\gamma_3(\theta,x)\right].
\end{align*}
%For $p=2,$ the function \eqref{eq:contrast} simplifies as follows
%\begin{equation*}
%\texttt{l}_{2,i}(\theta)=\frac{(X_i-r_{k_0}(\Delta_n,X_{i-1},\theta))^2}{2\Delta_nc_{i-1}(\beta)}+\frac12\log c_{i-1}(\beta)
%\end{equation*}
%which coincides with \eqref{locga}.
\begin{remark}
It is worth to point out that by assumptions $A_3$ and $A_5$ emerge that $\de_j$ and $\ee_j,$ for all $j\leq k_0,$ are three times differentiable with respect to $\theta.$ Furthermore, all their derivatives with respect to $\theta$ are of polynomial growth in $x$ uniformly in $\theta.$
\end{remark}

The contrast function \eqref{eq:contrast} yields to the maximum quasi-likelihood estimator $\hat\theta_{p,n}:=(\hat\alpha_{p,n},\hat\beta_{p,n})$ defined as 
\begin{equation}\label{eq:maxqlest}
l_{p,n}(\hat\theta_{p,n})=\inf_{\theta\in\Theta}l_{p,n}(\theta).
\end{equation}
Let $I(\theta_0)$ be the
Fisher information matrix at $\theta_0$ defined as follows
\begin{equation}
I(\theta_0):=\left(  \begin{matrix} % or pmatrix or bmatrix or Bmatrix or ...
    [I_b^{h,k}(\theta_0)]_{h,k=1,...,m_1} &   0 \\
  0 &  [I_\sigma^{h,k}(\theta_0)]_{h,k=1,...,m_2}\ \\
   \end{matrix}\right),
\end{equation}
where
\begin{align*}
I_b^{h,k}(\theta_0)&:=\int \left(\frac{\partial_{\alpha_h} b\,\partial_{\alpha_k} b}{c}\right)(\theta_0,x)\pi_0(\de x),\\
I_\sigma^{h,k}(\theta_0)&:=\frac12 \int \left(\frac{\partial_{\beta_h} c\,\partial_{\beta_k}c}{c^2}\right)(\beta_0,x)\pi_0(\de x).
\end{align*}

We recall an important asymptotic result which will be useful in the proof of our main theorem.

%\begin{lemma}\label{lemmaI}
%If $\Delta_n\to0$ and $n\Delta_n\to \infty,$ we have 
%\begin{equation}
%C_{n,p}(\theta_0)\overset{P_0}{\underset{n\to\infty}{\longrightarrow}}I(\theta_0) \tag{a}
%\end{equation}
%and
%\begin{equation}
%\sup_{|\theta|\leq \varepsilon_n}|C_{n,p}(\theta_0+\theta)-C_{n,p}(\theta_0)|\overset{P_0}{\underset{n\to\infty}{\longrightarrow}}0,\quad \varepsilon_n\to 0.\tag{b}
%\end{equation}

%\end{lemma}
\begin{theorem}[\cite{kess}]\label{teo:estconv}
Let $p$ be an integer and $k_0=[p/2].$ Under assumptions $A_1$ to $A_{4}, A_5[2k_0]$ and $A_6,$ if $\Delta_n\to0,n\Delta_n\to \infty,$ as $n\to\infty,$ the estimator $\hat\theta_{p,n}$ is consistent; i.e.
\begin{equation}\label{eq:cons}
\hat\theta_{p,n}\overset{P_0}{\underset{n\to\infty}{\longrightarrow}}\theta_0.
\end{equation}

If in addition $n\Delta_n^p\to 0$ and $\theta_0\in Int(\Theta)$ then
\begin{equation}\label{eq:conest}
\varphi(n)^{-1/2}(\hat\theta_{p,n}-\theta_0)=  \left( \begin{matrix} % or pmatrix or bmatrix or Bmatrix or ...
      \sqrt{n\Delta_n}(\hat\alpha_{p,n}-\alpha_0)\\
      \sqrt{n}(\hat\beta_{p,n}-\beta_0)    \\
   \end{matrix}\right)\overset{d}{\underset{n\to\infty}{\longrightarrow}}N_{m_1+m_2}(0,I^{-1}(\theta_0)),
\end{equation}
where $$\varphi(n):=\left( \begin{matrix} % or pmatrix or bmatrix or Bmatrix or ...
  \frac{1}{n\Delta_n}I_{m_1}&  0 \\
0&\frac{1}{n}I_{m_2} \\
   \end{matrix} \right).
$$
\end{theorem}

\begin{remark}
We observe that $l_{2,n}$ does not coincide with \eqref{locga}, because \eqref{eq:contrast} contains the terms $\de_1$ and $\ee_1.$ Nevertheless, $l_n$ also yields an asymptotical efficient estimator for $\theta$ and then we refer to it when $p=2.$
\end{remark}
\begin{remark}\label{ref:adapt}
Under the same framework adopted in this paper, alternatively to $\hat\theta_{p,n}$, \cite{kessad} and \cite{uchyos12} proposed different types of adaptive maximum quasi-likelihood estimators. 
For instance, in  \cite{uchyos12}, the first type of adaptive estimator is introduced starting from the initial estimator $\tilde \beta_{0,n}$ is defined by $\mathbb U_n(\tilde \beta_{0,n})=\inf_{\beta\in\Theta_\beta}\mathbb U_n(\beta),$ where $$\mathbb U_n(\beta):=\frac12\sum_{i=1}^n\left\{\frac{(X_i-X_{i-1})^2}{\Delta_n c_{i-1}(\beta)}+\log  c_{i-1}(\beta)\right\}.$$ For $p\geq 2, k_0=[p/2]$ and $l_0=[(p-1)/2],$ the first type adaptive estimator $\tilde\theta_{p,n}=(\tilde \alpha_{k_0,n},\tilde\beta_{l_0,n})$ is defined for $k=1,2,...,k_0,$ as follows
\begin{align*}
&l_{p,n}(\tilde \alpha_{k,n},\tilde \beta_{k-1,n})=\inf_{\alpha\in\Theta_\alpha}l_{p,n}(\alpha,\tilde\beta_{k-1,n}),\\
&l_{p,n}(\tilde \alpha_{k,n},\tilde \beta_{k,n})=\inf_{\beta\in\Theta_\beta}l_{p,n}(\tilde \alpha_{k,n},\beta).\end{align*}
 The maximum quasi-likelihood estimator $\hat\theta_{p,n}$ and its adaptive versions, like $\tilde\theta_{p,n},$ are asymptotically equivalent (under a minor change of the initial assumptions); i.e. they have the same properties \eqref{eq:cons} and \eqref{eq:conest} (see \cite{uchyos12}). In what follow we will developed a test based on $\hat\theta_{p,n};$ nevertheless in light of the previous discussion, it would be possible to replace $\hat\theta_{p,n}$ with $\tilde\theta_{p,n}.$
\end{remark}

\section{Test statistics}\label{test}

The goal of this section is to define and to analyze test statistics for the following parametric hypotheses problem
\begin{equation}\label{hp}H_0:\theta=\theta_0,\quad \text{vs} \quad H_1:\theta\neq \theta_0,\end{equation}
concerning the stochastic differential equation \eqref{sde}. $X$ is partially observed and therefore we have discrete observations represented by ${\bf X}_n$. The motivation of this research is due to the fact that under non-simple alternative hypotheses do not exist uniformly most powerful parametric tests. Therefore, we need proper procedure for making the right decision concerning statistical hypothesis.

The first step consists in the introduction of a suitable measure regarding the ``discrepancy'', or the ``distance'', between diffusions belonging to the parametric class \eqref{sde}. Furthermore, we bearing in mind that as recalled in the previous section, for a general stochastic differential equation $X,$ the true probability transitions from $X_{i-1}$ to $X_i$ do not exist in closed form as well as the likelihood function. 
Suppose known the parameter $\beta$ and assume observable the sample path up to time $T=n\Delta_n.$
Let $Q_\beta$ be the probability law of the process solution to $\de Y_t=\sigma(\beta,Y_t)\de W_t.$ The continuous loglikelihood of $X$ is given by
\begin{equation*}
\log \frac{\de P_\theta}{\de Q_\beta}=\int_0^T\frac{b(\alpha,X_t)}{c(\beta,X_t)}\de X_t-\frac12\int_0^T\frac{b^2(\alpha,X_t)}{c(\beta,X_t)}\de t.
\end{equation*}
Thus we can consider the (squared) $L^2(Q_\beta)$-distance between the loglikelihoods $\log \frac{\de P_{\theta_1}}{\de Q_\beta}$ and $\log\frac{\de P_{\theta_2}}{\de Q_\beta}$ with $\theta_1,\theta_2\in\Theta$; that is
\begin{equation}\label{eq:dloglik}
D(\theta_1,\theta_2):=\left|\left|\log \frac{\de P_{\theta_1}}{\de Q_\beta}-\log\frac{\de P_{\theta_2}}{\de Q_\beta}\right|\right|_{L^2(Q_\beta)}^2=\int \left[\log \frac{\de P_{\theta_1}}{\de Q_\beta}-\log\frac{\de P_{\theta_2}}{\de Q_\beta}\right]^2 \de Q_\beta.
\end{equation}
Clearly for testing the hypotheses \eqref{hp} in the framework of discretely observed stochastic differential equations, the distance \eqref{eq:dloglik} is not useful. Nevertheless, the above $L^2-$metric for the continuos observations suggests to consider 
\begin{equation}\label{eq: empdist}
\mathbb D_{p,n}(\theta_1,\theta_2):=\frac1n\sum_{i=1}^n[\texttt{l}_{p,i}(\theta_1)-\texttt{l}_{p,i}(\theta_2)]^2,\quad \theta_1,\theta_2\in\Theta,
\end{equation}
which can be interpreted as the empirical version of \eqref{eq:dloglik}, where the theoretical loglikelihood is replaced by the quasi-loglikelihood defined by \eqref{eq:contrast}. The following theorem provides the convergence in probability of $\mathbb D_{p,n}.$
\begin{theorem}\label{main4teo}
 Let $p$ be an integer and $k_0=[p/2].$  Assume $ A_1- A_{4}, A_5[2k_0]$ and $A_6.$ Under $H_0,$ if $\Delta_n\to0,n\Delta_n\to\infty,$ as $n\to \infty$,  we have
that 
\begin{equation*}
\mathbb D_{p,n}(\theta,\theta_0)\overset{P_0}{\underset{n\to\infty}{\longrightarrow}} U(\beta,\beta_0)
\end{equation*}
uniformly in $\theta,$ where
\begin{align*}&U(\beta,\beta_0)\\
&:=\frac14\int  \left\{3\left[\frac{c(\beta_0,x)}{c(\beta,x)}-1\right]^2+\left[\log\left(\frac{c(\beta,x)}{c(\beta_0,x)}\right)\right]^2+2\left[\frac{c(\beta_0,x)}{c(\beta,x)}-1\right]\log\left(\frac{c(\beta,x)}{c(\beta_0,x)}\right)\right\}\pi_0(\de x).
\end{align*}
\end{theorem}

The above result shows that $\mathbb D_{p,n}(\theta,\theta_0)$ is not a true approximation of $D_{p,n}(\theta,\theta_0)$ because it does not converge to $\int \left[\log(\pi_\theta(\de x)/\pi_0(\de x))\right]^2\pi_0(\de x).$ Nevertheless, the function \eqref{eq: empdist} allows to construct the main object of interest of the paper. Let $\hat\theta_n$ be the maximum quasi-likelihood estimator defined by \eqref{eq:maxqlest}, for testing the hypotheses \eqref{hp} we introduce the following class of test statistics
\begin{equation}\label{eq:famtest}
T_{p,n}(\hat\theta_{p,n},\theta_0):=n\mathbb D_{p,n}(\hat\theta_{p,n},\theta_0).
\end{equation}  
The first result concerns the weak convergence of $T_{p,n}(\hat\theta_{p,n},\theta_0).$ We prove that $T_{p,n}(\hat\theta_{p,n},\theta_0)$ is asymptotically distribution free under $H_0;$ namely it weakly converges to a chi-squared random variable with two degrees of freedom.

 \begin{theorem}\label{main}
Let $p$ be an integer and $k_0=[p/2].$ Assume $ A_1- A_{4}, A_5[2k_0]$ and $A_6.$ Under $H_0,$ if $\Delta_n\to0,n\Delta_n\to\infty, n\Delta_n^p\to 0,$ as $n\to \infty$,  we have
that
\begin{equation}\label{main2}
T_{p,n}(\hat\theta_{p,n},\theta_0)\overset{d}{\underset{n\to\infty}{\longrightarrow}} \chi_{m_1+m_2}^2.
\end{equation}
\end{theorem}

Given the level $\alpha\in(0,1)$, our criterion suggests to $$\text{reject}\, H_0\, \text{if }\, T_{p,n}(\hat\theta_{p,n},\theta_0)>\chi^2_{m_1+m_2,{\alpha}} ,$$ where $\chi^2_{m_1+m_2,{\alpha}}$ is the $1-\alpha$ quantile of the limiting random variable $\chi_{m_1+m_2}^2$; that is under $H_0$$$\lim_{n\to\infty}P_\theta(T_{p,n}(\hat\theta_{p,n},\theta_0)>\chi^2_{m_1+m_2,{\alpha}})=\alpha.$$

Under $H_1,$ the power function of the proposed test are equal to the following map
$$\theta\mapsto P_\theta\left(T_{p,n}(\hat\theta_{p,n},\theta_0)>\chi^2_{m_1+m_2,{\alpha}}\right)$$

Often a way to judge the quality of sequences of tests is provided by the powers at alternatives that become closer and closer to the null hypothesis. This justify the study of local limiting power. 
Indeed, usually the power functions of test statistic \eqref{eq:famtest} cannot be calculated explicitly. Nevertheless, $ P_\theta\left(T_{p,n}(\hat\theta_{p,n},\theta_0)>\chi^2_{m_1+m_2,{\alpha}}\right)$ can be studied and approximated under contiguous alternatives written as 
 \begin{equation}
 H_{1,n}:\theta=\theta_0+\varphi(n)^{1/2}h,
 \end{equation} 
 where $h\in \mathbb{R}^{m_1+m_2}$ such that $\theta_0+\varphi(n)^{1/2}h \in \Theta.$ In order to get a reasonable approximation of the power function, we analyze the asymptotic law of the test statistics under the local alternatives $H_{1,n}.$ We need the following assumption on the contiguity of probability measures (see \cite{van}):
 \begin{itemize}
 \item[$B_1.$] $P_{\theta_0+\varphi(n)h}$ is a sequence of contiguous probability measures with respect to $P_0;$  i.e. $ \lim_{n\to\infty} P_{0}(A_n)=0$ implies $\lim_{n\to\infty} P_{\theta_0+\varphi(n)^{1/2}h}(A_n)=0$ for every measurable sets $A_n$.
 \end{itemize}
 
 \begin{remark}
The assumption $B_1$ holds if we assume $A_1- A_{4}, A_5[2k_0]$ and the conditions:
\begin{itemize}
\item[(i)] there exists a constant $C>0$ such that the following estimates hold
$$|b(\alpha,x)|\leq C(1+|x|),\quad\left |\frac{\partial}{\partial x}b(\alpha,x)\right|+|\sigma(\beta,x)|+\left|\frac{\partial}{\partial x}\sigma(\beta,x)\right|\leq C$$
for all $(\alpha,\beta)\in\Theta$ and $x\in\mathbb R;$
\item[(ii)] there exists $C_0>0$ and $K>0$ such that
$$b(\alpha,x)x\leq -C_0|x|^2+K$$
 for all $(\alpha,x)\in\Theta_\alpha\times\mathbb R;$
 \item[(iii)] there exists a constant $C_1>0$ such that
 $$\frac{1}{C_1}\leq \sigma(\beta,x)\leq C_1.$$
\end{itemize}
Under the above assumptions, \cite{gobet} proved the Local Asymptotic Normality (LAN) for the likelihood of the ergodic diffusions \eqref{sde}; i.e.
$$\log\left(\frac{\de P_{\theta_0+\varphi(n)h}}{\de P_0}({\bf X}_n)\right)\overset{d}{\underset{n\to\infty}{\longrightarrow}} h' N_{m_1+m_2}(0, I(\theta_0))+\frac12 h' I(\theta_0)h.$$
By means of Le Cam's first lemma (see \cite{van}), LAN property implies the contiguity of $P_{\theta_0+\varphi(n)h}$ with respect to $P_0.$
 \end{remark}

Now, we are able to study the asymptotic probability distribution of $T_{p,n}$ under $H_{1,n}.$

  \begin{theorem}\label{main3teo}
 Let $p$ be an integer and $k_0=[p/2].$ Assume $A_1- A_{4}, A_5[2k_0], A_6$ and $B_1$ fulfill. Under the local alternative hypothesis $H_{1,n},$
 if $\Delta_n\to0,n\Delta_n\to\infty, n\Delta_n^p\to 0$ as $n\to \infty$, the following weak convergence holds
 \begin{equation}\label{main3}
T_{p,n}(\hat\theta_{p,n},\theta_0)\overset{d}{\underset{n\to\infty}{\longrightarrow}} \chi_{m_1+m_2}^2(h'I(\theta_0)h),
\end{equation}
where the random variable $\chi^2_{l+m}(h'I(\theta_0)h)$ is a non-central chi square random variable with $l+m$ degrees of freedom and non-centrality parameter $h'I(\theta_0)h$.
\end{theorem}
\begin{remark}
If we deal with $H_0 : \theta=\theta_0$ and the local alternative hypothesis $H_{1,n},$ Theorem \ref{main3teo} leads to the following approximation of the power functions
\begin{equation}
P_\theta\left(T_{p,n}(\hat\theta_{p,n},\theta_0)>\chi^2_{m_1+m_2,{\alpha}}\right)\cong
1-\mathbf{F}\left(\chi^2_{m_1+m_2,\alpha}\right),\quad n>>1,
\end{equation}
where $\mathbf{F}(\cdot)$ is the cumulative function
of the random variable $\chi^2_{m_1+m_2}(h'I(\theta_0)h)$.
\end{remark}

%A reasonable property for a sequence of tests is the consistency which allows to get the ``perfect'' pointwise limiting power function.

\begin{remark}
The Generalized Quasi-Likelihood Ratio, Wald, Rao type test statistics have been studied by \cite{kituch}, respectively, given by 
\begin{equation}
\label{GRLT}
L_{p,n}(\hat\theta_{p,n},\theta_0):= 2(l_{p,n}(\hat\theta_{p,n})-l_{p,n}(\theta_0))
\end{equation}
\begin{equation}
\label{WALD}
W_{p,n}(\hat\theta_{p,n},\theta_0):= (\varphi(n)^{-1/2}(\hat\theta_{p,n}-\theta_0))' I_{p,n}(\hat\theta_{p,n})\varphi(n)^{-1/2}(\hat\theta_{p,n}-\theta_0)
\end{equation}
\begin{equation}
\label{RAO}
R_{p,n}(\hat\theta_{p,n},\theta_0):=  (\varphi(n)^{1/2}\partial_\theta l_{p,n}(\theta_0))' I_{p,n}^{-1}(\hat\theta_{p,n})\varphi(n)^{1/2}\partial_\theta l_{p,n}(\theta_0),
\end{equation}
where 
$$ I_{p,n}(\theta)=\left(\begin{array}{cc }
  \frac{1}{n\Delta_n}\partial_\alpha^2l_{p,n}(\theta)    & \frac{1}{n\sqrt{\Delta_n}}\partial_{\alpha}\partial_{\beta} l_{p,n}(\theta)   \\
     \frac{1}{n\sqrt{\Delta_n}}\partial_{\beta}\partial_{\alpha} l_{p,n}(\theta)   &    \frac{1}{n}\partial_\beta^2l_{p,n}(\theta) 
\end{array}\right)
$$
and $R_{p,n}$ is well-defined if $ I_{p,n}(\theta)$ is nonsingular.
The above test statistics are asymptotically equivalent to $T_{p,n};$ i.e. under $H_0,$ $L_{p,n},W_{p,n}$ and $R_{p,n}$ weakly converge to a $\chi^2$ random variable.
\end{remark}

\begin{remark}
In \cite{DGItest}, the authors dealt with (for $p=2$) test statistics based on an empirical version of the true $\phi$-divergences; i.e.
\begin{equation}
2\sum_{i=1}^n\phi\left(\frac{\exp l_n(\theta)}{\exp  l_n(\theta_0)}{}\right)
\label{phidiv}
\end{equation}
where $\phi$ represents a suitable convex function and $l_n$ is given by \eqref{locga}.
In the present paper, the starting point is represented by the $L^2$-distance between two diffusion parametric models. Somehow, the approach developed in this work is close to that developed by \cite{aittest}, where a test based on the $L^2$-distance measure between the density function and its nonparametric estimator is introduced. 
\end{remark}

\begin{remark}
From a practical point of view, since sometimes $\alpha= \alpha_0$ and $\beta= \beta_0$ have different meanings, it is possible to deal with a stepwise procedure. For instance as $p=2,$ first, we test $\beta= \beta_0$ by means of 
$$T_n^\beta(\tilde\beta_{0,n},\beta_0):=\sum_{i=1}^n\left[\frac{(X_i-X_{i-1})^2}{\Delta_n}\left(\frac{1}{ c_{i-1}(\tilde\beta_{0,n})}-\frac{1}{ c_{i-1}(\beta_0)}\right)+\log \left( \frac{c_{i-1}(\tilde\beta_{0,n})}{c_{i-1}(\beta_0)}\right)\right]^2$$
and then, in the second step, we test  $\alpha= \alpha_0$ by taking into account
 $$T_n^\alpha(\tilde\alpha_{1,n},\alpha_0,\tilde\beta_{0,n}):=\sum_{i=1}^n[\texttt{l}_{2,i}(\tilde\alpha_{1,n}, \tilde\beta_{0,n})-\texttt{l}_{2,i}(\alpha_0,\tilde\beta_{0,n})]^2,$$
 where $\tilde\alpha_{1,n}$ and $\tilde\beta_{0,n}$ are the adaptive estimators defined in the Remark \ref{ref:adapt}.
 \end{remark}

\section{Numerical analysis}\label{sec:numerics}

Although all test statistics presented in the above and in the literature satisfy  the
same asymptotic results, for small sample sizes the performance of each test statistic is determined by the statistical model generating the data and the quality of the approximation of the quasi-likelihood function.
To put in evidence these effects we consider the two stochastic models presented in Section \ref{not}, namely the Ornstein-Uhlenbeck (OU in the tables) of equation \eqref{ou} and the CIR model of equation \eqref{cir}.
In this numerical study we  consider the power of the test under local alternatives for different test statistics:
\begin{itemize}
\item the $\phi$ divergence of equation \eqref{phidiv} with
$\phi(x) = 1-x+x \log(x)$, which is equivalent to the approximated Kullback-Leibler divergence (see, \cite{DGItest}). We use the label $AKL$ in the tables for this approximate KL;
\item the  $\phi$ divergence with $\phi(x) = \left(\frac{x-1}{x+1}\right)^2$: this was proposed in \cite{BS68}, we name it BS in the tables;
\item the Generalized Quasi-Likelihood Ratio test with $p=2$, see e.g., \eqref{GRLT}, denoted as GQLRT in the tables;
\item the Rao test statistics\footnote{We do not consider the Wald test of \eqref{WALD} because it was shown in \cite{kituch} that it performs similarly to the Rao test statistics.} $R(\hat\theta_{p,n},\theta_0)$ of equation \eqref{RAO}, denoted as RAO in the tables;
\item and the statistic $T_{p,n}(\hat\theta_{p,n},\theta_0)$ proposed in this paper and defined in equation \eqref{eq:famtest}, with $p=2$, denoted as $T_{2,n}$ in the tables.
\end{itemize}
The sample sizes have been chosen to be equal to $n=50, 100, 250, 500, 1000$ observations and time horizon is set to $T=n^\frac13$, in order to satisfy the asymptotic theory. For testing $\theta_0$ against the local alternatives $\theta_0 + \frac{h}{\sqrt{n\Delta_n}}$ for the parameters in the drift coefficient and $\theta_0 + \frac{h}{\sqrt{n}}$ for the parameters in the diffusion coefficient, $h$ is taken in a grid from $0$ to $1$, and $h=0$ corresponds to the null hypothesis $H_0$.
For the data generating process, we consider the following statistical models
\begin{itemize}
\item[OU:]  the one-dimensional Ornstein-Uhlenbeck model solution to  $\de X_t = \alpha_1(\alpha_2- X_t)\de t + \beta_1 \de W_t$, $X_0=1$, with $\theta_0=(\alpha_1, \alpha_2, \beta_1) = (0.5, 0.5, 0.25)$;
\item[CIR:]  the one-dimensional CIR model solution to  $\de X_t = \alpha_1(\alpha_2- X_t)\de t + \beta_1 \sqrt{X_t} \de W_t$, $X_0=1$, with $\theta_0=(\alpha_1, \alpha_2,\beta_1) = (0.5, 0.5, 0.125)$.
\end{itemize}

In each experiments the process have been simulated at high frequency using the Euler-Maruyama scheme and resampled to obtain $n=50, 100, 250, 500, 1000$ observations. 
Remark that, even if the Ornstein-Uhlenbeck process has a Gaussian transition density, this density is different from the Euler-Maruyama Gaussian density for non negligible time mesh $\Delta_n$ (see, \cite{iacus}).
For the simulation we user the R package {\it yuima} (see, \cite{iacusyosyui}).
Each experiment is replicated 1000 times and from the empirical distribution of each test statistic, say $S_n$, we define the rejection threshold of the test as $\tilde \chi^2_{3,0.05}$, i.e. $\tilde \chi^2_{3,0.05}$ is the 95\% quantile of the empirical distribution of $S_n,$ that is
$$
0.05 = \text{Freq}(S_n(\hat\theta_n, \theta_0) > \tilde \chi^2_{3,0.05}).
$$
Similarly, we define the empirical power function of the test as
$$
{\rm EPow}(h) =\text{Freq}(S_n(\hat\theta_n, \theta_0+\varphi(n)^{1/2}h) > \tilde \chi^2_{3,0.05}),
$$
where $\hat\theta_n$ is the maximum quasi-likelihood estimator defined in \eqref{eq:maxqlest}.
The choice of using the empirical threshold $\tilde \chi^2_{3,0.05}$ instead of the theoretical threshold $ \chi^2_{3,0.05}$ from the $\chi^2_3$ distribution, is due to the fact that otherwise the tests are non comparable. Indeed, the empirical level of the test is not $0.05$ for small sample sizes when $\chi^2_{3,0.05}$ is used as rejection threshold and, for example, when $h=0$ different choices of the test statistic produce different empirical levels of the test. 
Tables \ref{tab:OU} and \ref{tab:CIR} contain the empirical power function of each test.
In these tables  the bold face font is used to put in evidence the test statistics with the highest empirical power function ${\rm EPow}(h)$ for a given local alternative $h>0$.
As mentioned before, the natural benchmark test statistics is the generalised quasi likelihood ratio test (GQLRT).

From this numerical analysis we can see several facts:
\begin{itemize} 
\item the test statistic based on the AKL  test statistics does not perform as the GQLR test despite they are related to the same divergence; the latter being sometimes better;
\item the $T_{2,n}$ seems to be (almost) uniformly more powerful in this experiment;
\item all but RAO test seem to have a good behaviour when the alternative is sufficiently large;
\item for the CIR model, the RAO test does not perform well under the alternative hypothesis and this is probably because it requires very large $T$ which, in our case, is at most $T=10$. For the OU Gaussian case, the performance are better and in line from those presented in \cite{kituch} for similar sample sizes.
\end{itemize}
Therefore, we can conclude that, despite all the test statistics share the same asymptotic properties, the proposed $T_{p,n}$ seems to perform very well in the small sample case examined in the above Monte Carlo experiments, at least for $p=2$.

%\begin{landscape}
\begin{table}[ht]
\caption{Empirical power function ${\rm EPow}(h)$, for different sample sizes $n$ and local alternatives $h$. The empirical power and theoretical power is $0.05$. Data generating model: the 1-dimensional Ornstein-Uhlenbeck process.}
\label{tab:OU}
\begin{center}
{\tiny 
\begin{tabular}{c c}
$n=50$ & $n=100$\\
\begin{tabular}{rrrrrr}
  \hline
 & AKL & GQLRT & BS & RAO & $T_{2,n}$ \\ 
  \hline
h=0.00 & 0.050 & 0.050 & 0.050 & 0.050 & 0.050  \\ 
  h=0.01 & 0.044 &  0.048 & 0.046 & \bf 0.053 & 0.052  \\ 
  h=0.05 & 0.035 &   0.032 & 0.041 & \bf 0.057 & \bf 0.057  \\ 
  h=0.10 & 0.025 &   0.029 & 0.033 & 0.064 & \bf 0.077  \\ 
  h=0.20 & 0.011 &   0.031 & 0.042 & 0.078 & \bf 0.133  \\ 
  h=0.30 & 0.007 &    0.054 & 0.069 & 0.096 & \bf 0.239  \\ 
  h=0.40 & 0.007 &    0.108 & 0.147 & 0.121 & \bf 0.371  \\ 
  h=0.50 & 0.009 &    0.216 & 0.269 & 0.138 & \bf 0.559  \\ 
  h=0.60 & 0.021 &    0.359 & 0.448 & 0.146 & \bf 0.720  \\ 
  h=0.70 & 0.053 &    0.527 & 0.591 & 0.149 & \bf 0.842  \\ 
  h=0.80 & 0.120 &    0.670 & 0.736 & 0.150 & \bf 0.917  \\ 
  h=0.90 & 0.221 &    0.794 & 0.852 & 0.148 & \bf 0.966  \\ 
  h=1.00 & 0.383 &    0.882 & 0.910 & 0.145 & \bf 0.992  \\ 
   \hline
\end{tabular}
&
\begin{tabular}{rrrrrr}
  \hline
 & AKL & GQLRT & BS & RAO &  $T_{2,n}$ \\ 
  \hline
h=0.00 & 0.050 & 0.050 & 0.050 & 0.050 & 0.050  \\ 
  h=0.01 & 0.046 & 0.047 & 0.046 &  \bf 0.050 &  \bf  0.050  \\ 
  h=0.05 & 0.032 & 0.035 & 0.035 & 0.050 &\bf    0.055  \\ 
  h=0.10 & 0.022 & 0.029 & 0.030 & 0.058 & \bf   0.070  \\ 
  h=0.20 & 0.014 & 0.038 & 0.042 & 0.082 & \bf   0.141  \\ 
  h=0.30 & 0.009 & 0.089 & 0.083 & 0.101 & \bf   0.253  \\ 
  h=0.40 & 0.009 & 0.159 & 0.163 & 0.128 &  \bf  0.404  \\ 
  h=0.50 & 0.020 & 0.283 & 0.291 & 0.155 & \bf   0.609  \\ 
  h=0.60 & 0.051 &  0.465 & 0.472 & 0.183 &\bf  0.769 \\ 
  h=0.70 & 0.131 &  0.644 & 0.659 & 0.199 &\bf  0.876  \\ 
  h=0.80 & 0.244 &  0.789 & 0.801 & 0.213 &\bf  0.943  \\ 
  h=0.90 & 0.414 &  0.883 & 0.893 & 0.221 & \bf 0.984  \\ 
  h=1.00 & 0.608 &  0.937 & 0.944 & 0.225 & \bf 0.996  \\ 
   \hline
\end{tabular}\\
\\
$n=250$ & $n=500$\\
\begin{tabular}{rrrrrr}
  \hline
 & AKL & GQLRT & BS & RAO & $T_{2,n}$ \\ 
  \hline
h=0.00 & 0.050 & 0.050 & 0.050 & 0.050 & 0.050  \\ 
  h=0.01 & 0.044 & 0.049 & 0.050 & \bf 0.051 &  0.048  \\ 
  h=0.05 & 0.036 & 0.049 & 0.046 & 0.052 &\bf   0.057  \\ 
  h=0.10 & 0.028 & 0.048 & 0.050 & 0.058 &\bf   0.075  \\ 
  h=0.20 & 0.015 & 0.076 & 0.078 & 0.114 & \bf  0.143  \\ 
  h=0.30 & 0.022 & 0.153 & 0.157 & 0.168 & \bf  0.255 \\ 
  h=0.40 & 0.049 & 0.304 & 0.304 & 0.222 &\bf   0.452  \\ 
  h=0.50 & 0.118 & 0.486 & 0.496 & 0.280 &\bf   0.654  \\ 
  h=0.60 & 0.253 & 0.703 & 0.704 & 0.339 & \bf  0.822  \\ 
  h=0.70 & 0.436 &    0.847 & 0.851 & 0.389 & \bf  0.921  \\ 
  h=0.80 & 0.666 &    0.928 & 0.931 & 0.419 &  \bf 0.969 \\ 
  h=0.90 & 0.821 &    0.973 & 0.976 & 0.462 &  \bf 0.991  \\ 
  h=1.00 & 0.911 & 0.992 & 0.993 & 0.485 &\bf   1.000  \\ 
   \hline
\end{tabular}
&
\begin{tabular}{rrrrrr}
  \hline
 & AKL & GQLRT & BS & RAO & $T_{2,n}$ \\ 
  \hline
h=0.00 & 0.050 & 0.050 & 0.050 & 0.050 & 0.050  \\ 
  h=0.01 & 0.048 & 0.049 & 0.049 & \bf 0.052 &   0.051  \\ 
  h=0.05 & 0.038 & 0.044 & 0.043 &\bf 0.067 &   0.059  \\ 
  h=0.10 & 0.032 & 0.050 & 0.050 & \bf 0.082 &   0.075  \\ 
  h=0.20 & 0.030 & 0.084 & 0.080 & \bf 0.134 &   0.133  \\ 
  h=0.30 & 0.050 & 0.175 & 0.175 & 0.202 &\bf   0.250  \\ 
  h=0.40 & 0.138 & 0.329 & 0.323 & 0.279 & \bf  0.449  \\ 
  h=0.50 & 0.274 & 0.555 & 0.552 & 0.363 &\bf   0.673  \\ 
  h=0.60 & 0.493 & 0.751 & 0.747 & 0.454 &\bf   0.828  \\ 
  h=0.70 & 0.704 & 0.869 & 0.869 & 0.522 &\bf   0.934  \\ 
  h=0.80 & 0.847 & 0.957 & 0.957 & 0.584 & \bf  0.983  \\ 
  h=0.90 & 0.936 & 0.987 & 0.987 & 0.630 & \bf  0.996  \\ 
  h=1.00 & 0.982 & 0.997 & 0.997 & 0.678 & \bf  0.998 \\ 
   \hline
\end{tabular}\\
\\
\end{tabular}
\\
$n=1000$\\
\begin{tabular}{rrrrrr}
  \hline
 & AKL & GQLRT & BS & RAO & $T_n$  \\ 
  \hline
h=0.00 & 0.050 & 0.050 & 0.050 & 0.050 & 0.050  \\ 
  h=0.01 & 0.046 & 0.049 & 0.050 &   \bf 0.051 &  \bf 0.051  \\ 
  h=0.05 & 0.038 & 0.046 & 0.049 &   0.056 &  \bf 0.058  \\ 
  h=0.10 & 0.035 & 0.056 & 0.062 &  0.062 &  \bf 0.074  \\ 
  h=0.20 & 0.061 & 0.104 & 0.109 &   0.121 & \bf  0.134  \\ 
  h=0.30 & 0.122 & 0.182 & 0.187 &  0.193 &  \bf 0.241  \\ 
  h=0.40 & 0.219 & 0.359 & 0.372 & 0.291 &  \bf 0.442  \\ 
  h=0.50 & 0.426 & 0.600 & 0.605 & 0.398 &  \bf 0.662  \\ 
  h=0.60 & 0.655 & 0.786 & 0.794 & 0.507 & \bf  0.840  \\ 
  h=0.70 & 0.821 & 0.912 & 0.914 & 0.596 &  \bf 0.942  \\ 
  h=0.80 & 0.930 & 0.969 & 0.972 & 0.665 &  \bf 0.985  \\ 
  h=0.90 & 0.978 & 0.993 & 0.993 &   0.711 & \bf  0.994  \\ 
  h=1.00 & 0.994 & 0.997 & 0.997 &   0.760 & \bf  0.998   \\ 
   \hline
\end{tabular}
}
\end{center}
\end{table}
%\end{landscape}

%\begin{landscape}
\begin{table}[ht]
\caption{Empirical power function ${\rm EPow}(h)$, for different sample sizes $n$ and local alternatives $h$. The empirical power and theoretical power is $0.05$. Data generating model: the 1-dimensional CIR process.}
\label{tab:CIR}
\begin{center}
{\tiny 
\begin{tabular}{c c}
$n=50$ & $n=100$\\
\begin{tabular}{rrrrrr}
  \hline
 & AKL & GQLRT & BS & RAO & $T_{2,n}$ \\ 
  \hline
h=0.00 & 0.050 & 0.050 & 0.050 & 0.050 & 0.050  \\ 
  h=0.01 & 0.041 & 0.044 & 0.045 & 0.052 &  \bf  0.053  \\ 
  h=0.05 & 0.025 &    0.032 & 0.031 & 0.059 & 0.\bf 071  \\ 
  h=0.10 & 0.009 &    0.040 & 0.042 & 0.068 & \bf 0.145  \\ 
  h=0.20 & 0.013 &    0.148 & 0.167 & 0.075 & \bf 0.371  \\ 
  h=0.30 & 0.044 &    0.416 & 0.458 & 0.069 &\bf  0.721  \\ 
  h=0.40 & 0.186 &    0.700 & 0.741 & 0.067 &\bf  0.923  \\ 
  h=0.50 & 0.475 &    0.883 & 0.907 & 0.067 &\bf  0.989  \\ 
  h=0.60 & 0.760 &    0.967 & 0.981 & 0.061 &\bf  0.997  \\ 
  h=0.70 & 0.913 &    0.994 & 0.998 & 0.059 & \bf 1.000  \\ 
  h=0.80 & 0.981 &   \bf  1.000 &  \bf   1.000 & 0.051 &\bf  1.000  \\ 
  h=0.90 & 0.997 &   \bf  1.000 &  \bf   1.000 & 0.041 & \bf 1.000  \\ 
  h=1.00 &  \bf 1.000 &   \bf  1.000 &    \bf 1.000 & 0.041 & \bf 1.000  \\ 
   \hline
\end{tabular}
&
\begin{tabular}{rrrrrr}
  \hline
 & AKL & GQLRT & BS & RAO & $T_{2,n}$ \\ 
  \hline
h=0.00 & 0.050 & 0.050 & 0.050 & 0.050 & 0.050  \\ 
  h=0.01 & 0.040 & 0.043 & 0.046 & \bf 0.053 &\bf  0.051  \\ 
  h=0.05 & 0.019 & 0.032 & 0.034 & 0.056 & \bf  0.070  \\ 
  h=0.10 & 0.010 & 0.054 & 0.051 & 0.062 & \bf  0.150  \\ 
  h=0.20 & 0.017 & 0.205 & 0.207 & 0.063 & \bf  0.461  \\ 
  h=0.30 & 0.102 &    0.537 & 0.553 & 0.064 &\bf 0.797  \\ 
  h=0.40 & 0.338 &    0.827 & 0.836 & 0.064 &\bf 0.957  \\ 
  h=0.50 & 0.685 &    0.950 & 0.958 & 0.063 &\bf 0.995  \\ 
  h=0.60 & 0.896 &    0.993 & 0.994 & 0.059 & \bf 1.000  \\ 
  h=0.70 & 0.977 &    0.999 & 0.998 & 0.056 & \bf 1.000  \\ 
  h=0.80 & 0.998 & \bf  1.000 & \bf  1.000 & 0.053 & \bf 1.000   \\ 
  h=0.90 & 0.999 & \bf  1.000 & \bf  1.000 & 0.048 & \bf 1.000  \\ 
  h=1.00 & \bf 1.000 & \bf  1.000 & \bf  1.000 & 0.044 & 1.000   \\ 
    \hline
\end{tabular}
\\
\\
$n=250$ & $n=500$\\
\begin{tabular}{rrrrrr}
  \hline
 & AKL & GQLRT & BS & RAO & $T_{2,n}$ \\ 
  \hline
h=0.00 & 0.050 & 0.050 & 0.050 & 0.050 & 0.050  \\ 
  h=0.01 & 0.042 & 0.049 & 0.046 & \bf 0.052 &  0.050  \\ 
  h=0.05 & 0.026 & 0.045 & 0.046 & 0.054 & \bf  0.071  \\ 
  h=0.10 & 0.021 & 0.086 & 0.084 & 0.057 &  \bf 0.144  \\ 
  h=0.20 & 0.093 & 0.347 & 0.342 & 0.062 & \bf  0.505  \\ 
  h=0.30 & 0.372 & 0.752 & 0.756 & 0.064 &  \bf 0.864  \\ 
  h=0.40 & 0.790 &    0.943 & 0.944 & 0.065 & \bf 0.977  \\ 
  h=0.50 & 0.952 &    0.994 & 0.994 & 0.064 &  \bf 1.000  \\ 
  h=0.60 & 0.996 &  \bf 1.000 & \bf  1.000 &   0.060 &  \bf 1.000 \\ 
  h=0.70 &  \bf 1.000 &  \bf 1.000 &  \bf 1.000 &  0.060 &  \bf 1.000  \\ 
  h=0.80 &  \bf 1.000 &  \bf 1.000 &  \bf 1.000 &  0.057 &  \bf 1.000  \\ 
  h=0.90 &  \bf 1.000 &  \bf 1.000 &  \bf 1.000 &  0.055 &  \bf 1.000  \\ 
  h=1.00 &  \bf 1.000 &  \bf 1.000 &  \bf 1.000 &  0.050 &  \bf 1.000  \\ 
    \hline
\end{tabular}
&
\begin{tabular}{rrrrrr}
  \hline
 & AKL & GQLRT & BS & RAO & $T_{2,n}$ \\ 
  \hline
 h=0.00 & 0.050 & 0.050 & 0.050 & 0.050 & 0.050  \\ 
  h=0.01 & 0.043 & 0.043 & 0.042 &\bf  0.051 & 0.048  \\ 
  h=0.05 & 0.030 & 0.046 & 0.044 & 0.051 &\bf   0.074  \\ 
  h=0.10 & 0.032 & 0.095 & 0.091 & 0.052 & \bf  0.147  \\ 
  h=0.20 & 0.180 & 0.384 & 0.380 & 0.055 & \bf  0.530  \\ 
  h=0.30 & 0.598 & 0.802 & 0.800 & 0.058 & \bf  0.869  \\ 
  h=0.40 & 0.898 & 0.972 & 0.972 & 0.058 & \bf  0.990  \\ 
  h=0.50 & 0.992 &\bf 0.998 &\bf 0.998 & 0.059 & \bf  0.998  \\ 
  h=0.60 & 0.998 & \bf 0.999 & \bf 0.999 & 0.057 &\bf   0.999  \\ 
  h=0.70 &   0.999 & \bf  1.000 & \bf  1.000 &   0.056 & \bf  1.000  \\ 
  h=0.80 & \bf  1.000 & \bf  1.000 & \bf  1.000 & 0.055 & \bf  1.000  \\ 
  h=0.90 & \bf  1.000 & \bf  1.000 & \bf  1.000 & 0.055 & \bf  1.000  \\ 
  h=1.00 & \bf  1.000 & \bf  1.000 & \bf  1.000 & 0.051& \bf  1.000  \\ 
    \hline
\end{tabular}\\
\\
\end{tabular}\\
$n=1000$\\
\begin{tabular}{rrrrrrr}
  \hline
 & AKL & GQLRT & BS & RAO & $T_{2,n}$ \\ 
  \hline
h=0.00 & 0.050 & 0.050 & 0.050 & 0.050 & 0.050  \\ 
  h=0.01 & 0.044 & 0.048 & 0.047 & \bf 0.051 & 0.050  \\ 
  h=0.05 & 0.035 & 0.059 & 0.057 & 0.051 &\bf  0.079  \\ 
  h=0.10 & 0.067 & 0.120 & 0.118 &  0.054 & \bf 0.144  \\ 
  h=0.20 & 0.274 & 0.429 & 0.428 & 0.058 & \bf 0.527  \\ 
  h=0.30 & 0.725 & 0.844 & 0.840 & 0.061 & \bf 0.886  \\ 
  h=0.40 & 0.953 & 0.983 & 0.983 & 0.062 & \bf 0.989  \\ 
  h=0.50 & 0.996 & 0.998 &  0.998 & 0.062 & \bf 0.999  \\ 
  h=0.60 & \bf 1.000 & \bf 1.000 & \bf 1.000 & 0.062 & \bf 1.000  \\ 
  h=0.70 & \bf 1.000 & \bf 1.000 & \bf 1.000 & 0.060 & \bf 1.000  \\ 
  h=0.80 & \bf 1.000 & \bf 1.000 & \bf 1.000 & 0.059 & \bf 1.000  \\ 
  h=0.90 & \bf 1.000 & \bf 1.000 & \bf 1.000 & 0.059 & \bf 1.000  \\ 
  h=1.00 & \bf 1.000 & \bf 1.000 & \bf 1.000 & 0.058 & \bf 1.000  \\ 
   \hline
\end{tabular}
}
\end{center}
\end{table}
%\end{landscape}

\section{Proofs}\label{proofs}

In order to prove the theorems appearing in the paper, we need some preliminary results. Let us start with the following lemmas.

\begin{lemma}\label{lemma0}
For $k\geq 1$ and $t_{i-1}^n\leq t\leq t_i^n$
\begin{equation}\label{eq:lemm1}
E_0^{i-1}[|X_t-X_{i-1}|^k]\leq C_k|t-t_{i-1}^n|^{k/2}(1+|X_{i-1}|)^{C_k}.
\end{equation}
If $f:\Theta\times \mathbb R\to\mathbb R$ is of polynomial growth in $x$ uniformly in $\theta$ then
\begin{equation}\label{eq:lemm1bis}
E_0^{i-1}[f(\theta,X_{t})]\leq C_{t-t_{i-1}^n}(1+|X_{i-1}|)^{C},\quad t_{i-1}^n\leq t\leq t_i^n.
\end{equation}
\end{lemma}
\begin{proof}
See the proof of Lemma 6 in \cite{kess}.
\end{proof}
\begin{lemma}\label{eq:lemma1}
For $l\geq 1$
\begin{align}
&r_{l}(\Delta_n,X_{i-1},\theta)=X_{i-1}+\Delta_n b_{i-1}(\alpha)+R(\theta,\Delta_n^2, X_{i-1})\label{eq:proof1}\\
&E_{0}^{i-1}[(X_{i}-r_{l}(\Delta_n,X_{i-1},\theta))^2]=\Delta_n c_{i-1}(\beta_0)+R(\theta,\Delta_n^2, X_{i-1})\label{eq:proof2}\\
&E_{0}^{i-1}[(X_{i}-r_{l}(\Delta_n,X_{i-1},\theta))^3]=R(\theta,\Delta_n^2, X_{i-1})\label{eq:proof3}\\
&E_{0}^{i-1}[(X_{i}-r_{l}(\Delta_n,X_{i-1},\theta))^4]=3\Delta_n^2 c_{i-1}^2(\beta_0)+R(\theta,\Delta_n^3, X_{i-1})\label{eq:proof4}\\
&E_{0}^{i-1}[(X_{i}-r_{l}(\Delta_n,X_{i-1},\theta))^5]=R(\theta,\Delta_n^3, X_{i-1})\label{eq:proof5}\\
&E_{0}^{i-1}[(X_{i}-r_{l}(\Delta_n,X_{i-1},\theta))^6]=5\cdot 3\Delta_n^3 c_{i-1}^3(\beta_0)+R(\theta,\Delta_n^4, X_{i-1})\label{eq:proof6}\\
&E_{0}^{i-1}[(X_{i}-r_{l}(\Delta_n,X_{i-1},\theta))^7]=R(\theta,\Delta_n^4, X_{i-1})\label{eq:proof7}\\
&E_{0}^{i-1}[(X_{i}-r_{l}(\Delta_n,X_{i-1},\theta))^8]=7\cdot5\cdot 3\Delta_n^4 c_{i-1}^4(\beta_0)+R(\theta,\Delta_n^5, X_{i-1})\label{eq:proof8}
\end{align}
\end{lemma}
\begin{proof}
The equalities from \eqref{eq:proof1}  to \eqref{eq:proof4} represent the statement of Lemma 7 in \cite{kess}. By using the same approach adopted for the proof of the aforementioned lemma, we observe that from \eqref{eq:proof1} to \eqref{eq:proof4}, the result \eqref{eq:proof5} and \eqref{eq:proof6} hold, if we are able to show that 
\begin{align}
&E_{0}^{i-1}[(X_{i}-X_{i-1})^5]=R(\theta,\Delta_n^3, X_{i-1})\label{eq:proof5bis}\\
&E_{0}^{i-1}[(X_{i}-X_{i-1})^6]=5\cdot 3\Delta_n^3 c_{i-1}^3(\beta_0)+R(\theta,\Delta_n^4, X_{i-1})\label{eq:proof6bis}
\end{align}
We only prove \eqref{eq:proof6bis}, because \eqref{eq:proof5bis} follows by means of similar arguments. By applying the Ito-Taylor formula (see Lemma 1, in \cite{fz}) to the function $f_x(y)=(y-x)^6$ we obtain
\begin{align*}
E_0^{i-1}[(X_i-X_{i-1})^6]&=f_{X_{i-1}}(X_{i-1})+\Delta_n L_0f_{X_{i-1}}(X_{i-1})\\
&\quad+\frac{\Delta_n^2}{2}L_0^2f_{X_{i-1}}(X_{i-1})+\frac{\Delta_n^3}{3!}L_0^3f_{X_{i-1}}(X_{i-1})\\
&\quad+ \int_0^{\Delta_n}\int_0^{u_1}\int_0^{u_2}\int_0^{u_3}E_0^{i-1}[L_0^4 f_{X_{i-1}}(X_{t_{i-1}^n+u_4})]\de u_1\de u_2\de u_3\de u_4.
\end{align*}
By applying \eqref{eq:lemm1bis}, we obtain
$$ \int_0^{\Delta_n}\int_0^{u_1}\int_0^{u_2}\int_0^{u_3}E_0^{i-1}[L_0^4 f_{X_{i-1}}(X_{t_{i-1}^n+u_4})]\de u_1\de u_2\de u_3\de u_4=R(\theta,\Delta_n^4, X_{i-1}).$$
Furthermore, by means of long and cumbersome calculations, we can show that $f_x(x)=L_0f_x(x)=L_0^2f_x(x)=0,$ while $L_0^3f_x(x)=5\cdot3\cdot 3!\Delta_n^3 c_{i-1}^3(\beta_0).$

Analogously to what done, from   \eqref{eq:proof1} to \eqref{eq:proof6}, the equalities \eqref{eq:proof7} and \eqref{eq:proof8} hold, if we are able to show that 
\begin{align}
&E_{0}^{i-1}[(X_{i}-X_{i-1})^7]=R(\theta,\Delta_n^4, X_{i-1}),\label{eq:proof7bis}\\
&E_{0}^{i-1}[(X_{i}-X_{i-1})^8]=7\cdot5\cdot 3\Delta_n^4 c_{i-1}^4(\beta_0)+R(\theta,\Delta_n^5, X_{i-1}).\label{eq:proof8bis}
\end{align}
We only prove \eqref{eq:proof8bis}, because \eqref{eq:proof7bis} follows by means of similar arguments. 
The application of the Ito-Taylor formula to the function $f_x(y)=(y-x)^8$ yields
\begin{align*}
E_0^{i-1}[(X_i-X_{i-1})^8]&=f_{X_{i-1}}(X_{i-1})+\Delta_n L_0f_{X_{i-1}}(X_{i-1})+\frac{\Delta_n^2}{2}L_0^2f_{X_{i-1}}(X_{i-1})\\
&\quad +\frac{\Delta_n^3}{3!}L_0^3f_{X_{i-1}}(X_{i-1})+\frac{\Delta_n^4}{4!}L_0^4f_{X_{i-1}}(X_{i-1})\\
&\quad+ \int_0^{\Delta_n}\int_0^{u_1}\int_0^{u_2}\int_0^{u_3}\int_0^{u_4}E_0^{i-1}[L_0^5 f_{X_{i-1}}(X_{t_{i-1}^n+u_5})]\de u_1\de u_2\de u_3\de u_4\de u_5
\end{align*}
By applying \eqref{eq:lemm1bis}, we get 
$$ \int_0^{\Delta_n}\int_0^{u_1}\int_0^{u_2}\int_0^{u_3}\int_0^{u_4}E_0^{i-1}[L_0^5 f_{X_{i-1}}(X_{t_{i-1}^n+u_5})]\de u_1\de u_2\de u_3\de u_4\de u_5=R(\theta,\Delta_n^5, X_{i-1}).$$
Furthermore, by means of long and cumbersome calculations, we can show that $f_x(x)=L_0f_x(x)=L_0^2f_x(x)=L_0^3f_x(x)=0$ while $L_0^4f_x(x)=7\cdot5\cdot 3\cdot 4!\Delta_n^4 c^4(\beta_0,x).$
\end{proof}

\begin{lemma}[Triangular arrays convegence]\label{lemmaGJ}
Let $U_i^n$ and $U$ be random variables, with $U_i^n$ being $\mathcal{G}_{i}^n$-measurable. The two following conditions imply $\sum_{i=1}^nU_i^n\overset{P}{\underset{n\to\infty}{\longrightarrow}} U$:
 \begin{align*}
 &\sum_{i=1}^nE[U_i^n|\mathcal{G}_{i-1}^n]\overset{P}{\underset{n\to\infty}{\longrightarrow}} U,\quad \sum_{i=1}^nE[(U_i^n)^2|\mathcal{G}_{i-1}^n]\overset{P}{\underset{n\to\infty}{\longrightarrow}} 0
 \end{align*}
\end{lemma}
\begin{proof} See the proof of Lemma 9 in \cite{gcj}.
\end{proof}
\begin{lemma}\label{lemmak}
Let $f:\Theta\times\mathbb{R}\to \mathbb{R}$ be such that $f(\theta,x)\in C_{\uparrow}^{1,1}(\Theta\times \re,\re).$ Let us assume $A_1-A_6$, if $\Delta_n\to0$ and $n\Delta_n\to\infty$ we have that
$$\frac1n\sum_{i=1}^n f_{i-1}(\theta)\overset{P_0}{\underset{n\to\infty}{\longrightarrow}} \int f(x,\theta)\pi_{0}(\de x)$$
uniformly in $\theta$.
%; i.e. $$\sup_{\theta\in\Theta}\left|\frac1n\sum_{i=1}^n f(X_{i-1},\theta)-\int f(x,\theta)\pi_{0}(\de x)\right|\overset{P_0}{\underset{n\to\infty}{\longrightarrow}}0$$
\end{lemma}
\begin{proof} See the proof of Lemma 8 in \cite{kess}.

\end{proof}

\begin{lemma}\label{lemmaer}
Let $f:\Theta\times\mathbb{R}\to \mathbb{R}$ be such that $f(\theta,x)\in C_{\uparrow}^{1,1}(\Theta\times \re,\re).$ Let us assume $A_1-A_6$, if $\Delta_n\to0$ and $n\Delta_n\to\infty,$ as $n\to\infty,$ we have that
$$\frac{1}{n\Delta_n^j}\sum_{i=1}^n f_{i-1}(\theta)(X_{i}-r_{l}(\Delta_n,X_{i-1},\theta_0))^k\overset{P_0}{\underset{n\to\infty}{\longrightarrow}}
\begin{cases}0,& j=1,k=1,\\
\int f(\theta,x)c(\beta_0,x)\pi_{0}(\de x),& j=1,k=2,\\
\int f(\theta,x)R(\theta,1,x)\pi_0(\de x),& j=2,k=3,\\
0,& j=1,k=4,\\
3\int f(\theta,x)c^2(\beta_0,x)\pi_0(\de x),& j=2,k=4,
\end{cases}$$
uniformly in $\theta$.
\end{lemma}
\begin{proof}
The cases $j=1,k=1$ and $j=1,k=2$ coincide with Lemma 9 and Lemma 10 in \cite{kess} and then we use the same approach to show that remaining convergences hold true. 

By setting
 $$\zeta_i^n(\theta):=\frac{1}{n\Delta_n^2}f_{i-1}(\theta)(X_{i}-r_{l}(\Delta_n,X_{i-1},\theta_0))^3,$$
 we prove that the convergence holds for all $\theta.$ By taking into account Lemma \ref{eq:lemma1} 
 \begin{align*}
 &E_0^{i-1}[\zeta_i^n(\theta)]=\frac1n\sum_{i=1}^nf_{i-1}(\theta)R(\theta,1, X_{i-1})\overset{P_0}{\underset{n\to\infty}{\longrightarrow}}\int f(\theta,x)R(\theta,1,x)\pi_0(\de x),\\
 &E_0^{i-1}[(\zeta_i^n(\theta))^2]=\frac{1}{n^2\Delta_n}\sum_{i=1}^n[5\cdot 3 c_{i-1}^3(\beta_0)+R(\theta,1, X_{i-1})]\overset{P_0}{\underset{n\to\infty}{\longrightarrow}}0.
 \end{align*}
Therefore by Lemma \ref{lemmaGJ} we can conclude that
$$\zeta_i^n(\theta)\overset{P_0}{\underset{n\to\infty}{\longrightarrow}}\int f(\theta,x)R(\theta,1,x)\pi_0(\de x),$$
for all $\theta.$
 For the uniformity of the convergence we use the same arguments adopted in the proof of Lemma 8 in \cite{kess}. Hence, it is sufficient to prove the tightness of the sequence of random elements $$Y_n(\theta):=\frac1n\sum_{i=1}^n\frac{f_{i-1}(\theta)(X_{i}-r_{l}(\Delta_n,X_{i-1},\theta_0))^3}{\Delta_n^2}$$ taking values in the Banach space $C(\Theta)$ endowed with the sup-norm $||\cdot||_\infty.$  From the assumptions of lemma follows that $\sup_nE_0[\sup_{\theta\in\Theta}|\partial_\theta Y_n(\theta)|]<\infty$ which implies the tightness of $Y_n(\theta)$ for the criterion given by Theorem 16.5 in \cite{kall}. 
 
 By setting
 $$\zeta_i^n(\theta):=\frac{1}{n\Delta_n^2}f_{i-1}(\theta)(X_{i}-r_{l}(\Delta_n,X_{i-1},\theta_0))^4,$$
  we prove that the convergence holds for all $\theta.$
  By taking into account Lemma \ref{eq:lemma1} and Lemma \ref{lemmaer}
 \begin{align*}
 &E_0^{i-1}[\zeta_i^n(\theta)]=\frac1n\sum_{i=1}^nf_{i-1}(\theta)[3 c_{i-1}^2(\beta_0)+R(\theta,\Delta_n, X_{i-1})]\overset{P_0}{\underset{n\to\infty}{\longrightarrow}}3\int f(\theta,x)c^2(\beta_0,x)\pi_0(\de x),\\
 &E_0^{i-1}[(\zeta_i^n(\theta))^2]=\frac{1}{n^2}\sum_{i=1}^n[7\cdot 5\cdot 3 c_{i-1}^4(\beta_0)+R(\theta,\Delta_n, X_{i-1})]\overset{P_0}{\underset{n\to\infty}{\longrightarrow}}0.
 \end{align*}
Therefore by Lemma \ref{lemmaGJ} we get the pointwise convergence.
 For the uniformity of the convergence we proceed as done above.
\end{proof}

Before to proceed with the proofs of the main theorems of the paper, we introduce some useful quantities coinciding with (4.2)$-$(4.8) appearing in \cite{kess}. We can write down
\begin{align}\label{eq:diff}
\texttt{l}_{p,i}(\theta)-\texttt{l}_{p,i}(\theta_0)=\varphi_{i,1}(\theta,\theta_0)+\varphi_{i,2}(\theta,\theta_0)+\varphi_{i,3}(\theta,\theta_0)+\varphi_{i,4}(\theta,\theta_0),
\end{align}
where
\begin{align*}
\varphi_{i,1}(\theta,\theta_0)&:=\frac{(X_i-r_{k_0}(\Delta_n,X_{i-1},\theta_0))^2}{2\Delta_n}\left\{\frac{1+ \sum_{j=1}^{k_0}\Delta_n^j \de_j(\theta,X_{i-1})}{c_{i-1}(\beta)}-\frac{1+ \sum_{j=1}^{k_0}\Delta_n^j \de_j(\theta_0,X_{i-1})}{c_{i-1}(\beta_0)}\right\},\\
\varphi_{i,2}(\theta,\theta_0)&:=\frac{(X_i-r_{k_0}(\Delta_n,X_{i-1},\theta_0))(r_{k_0}(\Delta_n,X_{i-1},\theta_0)-r_{k_0}(\Delta_n,X_{i-1},\theta))}{\Delta_nc_{i-1}(\beta)}\\
&\quad\times\left\{1+ \sum_{j=1}^{k_0}\Delta_n^j \de_j(\theta,X_{i-1})\right\},\\
\varphi_{i,3}(\theta,\theta_0)&:=\frac{(r_{k_0}(\Delta_n,X_{i-1},\theta_0)-r_{k_0}(\Delta_n,X_{i-1},\theta))^2}{2\Delta_nc_{i-1}(\beta)}\left\{1+ \sum_{j=1}^{k_0}\Delta_n^j \de_j(\theta,X_{i-1})\right\},\\
\varphi_{i,4}(\theta,\theta_0)&:=\frac12\log\left(\frac{c_{i-1}(\beta)}{c_{i-1}(\beta_0)}\right)+ \frac12\sum_{j=1}^{k_0}\Delta_n^j (\ee_j(\theta,X_{i-1})-\ee_j(\theta_0,X_{i-1})).
\end{align*}

Furthermore
\begin{align}\label{eq:deralpha}
\partial_{\alpha_h} \texttt{l}_{p,i}(\theta)=\eta_{i,1}^h(\theta)+\eta_{i,2}^h(\theta), \quad h=1,2,...,m_1,
\end{align}
where
\begin{align*}
&\eta_{i,1}^h(\theta):=-(\partial_{\alpha_h} r_{k_0}(\Delta_n,X_{i-1},\theta))(X_i-r_{k_0}(\Delta_n,X_{i-1},\theta))\frac{\left\{1+\sum_{j=1}^{k_0}\Delta_n^j \de_j(\theta,X_{i-1})\right\}}{\Delta_nc_{i-1}(\beta)},\\
&\eta_{i,2}^h(\theta):=(X_i-r_{k_0}(\Delta_n,X_{i-1},\theta))^2\frac{\sum_{j=1}^{k_0}\Delta_n^j \partial_{\alpha_h}\de_j(\theta,X_{i-1})}{2\Delta_nc_{i-1}(\beta)}+\frac12\sum_{j=1}^{k_0}\Delta_n^j \partial_{\alpha_h}\ee_j(\theta,X_{i-1}),
\end{align*}
and
\begin{align}\label{eq:derbeta}
\partial_{\beta_k}\texttt{l}_{p,i}(\theta)=\xi_{i,1}^k(\theta)+\xi_{i,2}^k(\theta)+\xi_{i,3}^k(\theta), \quad k=1,2,...,m_2,
\end{align}
where
\begin{align*}
&\xi_{i,1}^k(\theta):=\frac{(X_i-r_{k_0}(\Delta_n,X_{i-1},\theta))^2}{2\Delta_nc_{i-1}(\beta)}\left\{\sum_{j=1}^{k_0}\Delta_n^j \partial_{\beta_k}\de_j(\theta,X_{i-1})\right\}+\frac12\sum_{j=1}^{k_0}\Delta_n^j \partial_{\beta_k}\ee_j(\theta,X_{i-1}),\\
&\xi_{i,2}^k(\theta):=-\frac{(X_i-r_{k_0}(\Delta_n,X_{i-1},\theta))^2\partial_{\beta_k} c_{i-1}(\beta)}{2\Delta_nc_{i-1}^2(\beta)}\left\{1+\sum_{j=1}^{k_0}\Delta_n^j \de_j(\theta,X_{i-1})\right\}+\frac{\partial_{\beta_k} c_{i-1}(\beta)}{2c_{i-1}(\beta)},\\
&\xi_{i,3}^k(\theta):=-(\partial_{\beta_k} r_{k_0}(\Delta_n,X_{i-1},\theta))(X_i-r_{k_0}(\Delta_n,X_{i-1},\theta))\frac{\left\{1+\sum_{j=1}^{k_0}\Delta_n^j \de_j(\theta,X_{i-1})\right\}}{\Delta_nc_{i-1}(\beta)}.
\end{align*}

From \eqref{eq:diff} it is possible to derive
\begin{align}\label{eq:deralphaalpha}
\partial_{\alpha_h\alpha_k}^2\texttt{l}_{p,i}(\theta):=\delta_{i,1}^{h,k}(\theta)+\delta_{i,2}^{h,k}(\theta)+\delta_{i,3}^{h,k}(\theta)+\delta_{i,4}^{h,k}(\theta),\quad h,k=1,2,...,m_1,\end{align}
where
\begin{align*}
\delta_{i,1}^{h,k}(\theta)&:=\frac{(X_i-r_{k_0}(\Delta_n,X_{i-1},\theta_0))^2}{2c_{i-1}(\beta)}\{(\partial_{\alpha_h\alpha_k}^2 \de_1)_{i-1}(\theta)+R(\theta,\Delta_n, X_{i-1})\},\\
\delta_{i,2}^{h,k}(\theta)&:= \frac{(X_i-r_{k_0}(\Delta_n,X_{i-1},\theta_0))}{c_{i-1}(\beta)}\{-\partial_{\alpha_h\alpha_k}^2 b_{i-1}(\alpha)+R(\theta,\Delta_n, X_{i-1})\},\\
\delta_{i,3}^{h,k}(\theta)&:=\frac12\Delta_n\partial_{\alpha_h\alpha_k}^2\ee_1(\theta,X_{i-1}),\notag\\
\delta_{i,4}^{h,k}(\theta)&:=\Delta_n\left\{\frac{\partial_{\alpha_h\alpha_k}^2 b_{i-1}(\alpha)(b_{i-1}(\alpha)-b_{i-1}(\alpha_0))+\partial_{\alpha_h} b_{i-1}(\alpha)\partial_{\alpha_k} b_{i-1}(\alpha)}{c_{i-1}(\beta)}+R(\theta,\Delta_n, X_{i-1})\right\},\notag
\end{align*}
\begin{align}
\partial_{\beta_h\beta_k}^2\texttt{l}_{p,i}(\theta):=\nu_{i,1}^{h,k}(\theta)+\nu_{i,2}^{h,k}(\theta)+\nu_{i,3}^{h,k}(\theta), \quad h,k=1,2,...,m_2,
\end{align}
where
\begin{align*}
\nu_{i,1}^{h,k}(\theta)&:=\frac{(X_i-r_{k_0}(\Delta_n,X_{i-1},\theta_0))^2}{2\Delta_n}\{(\partial_{\beta_h\beta_k}^2 c^{-1})_{i-1}(\beta)+R(\theta,\Delta_n, X_{i-1})\},\\
\nu_{i,2}^{h,k}(\theta)&:= \frac12(X_i-r_{k_0}(\Delta_n,X_{i-1},\theta_0))R(\theta,1, X_{i-1})),\\
\nu_{i,3}^{h,k}(\theta)&:=\frac12(\partial_{\beta_h\beta_k}^2\log c)_{i-1}(\beta)+R(\theta,\Delta_n, X_{i-1})),\notag
\end{align*}
and
\begin{align}\label{eq:dersecalphabeta}
\partial_{\alpha_h\beta_k}^2\texttt{l}_{p,i}(\theta):=\mu_{i,1}(\theta)+\mu_{i,2}(\theta),\quad h=1,2,...,m_1, k=1,2,...,m_2,
\end{align}
where
\begin{align*}
\mu_{i,1}(\theta)&:=\frac{(X_i-r_{k_0}(\Delta_n,X_{i-1},\theta_0))^2}{2\Delta_n}R(\theta,\Delta_n, X_{i-1}),\\
\mu_{i,2}(\theta)&:= \frac{(X_i-r_{k_0}(\Delta_n,X_{i-1},\theta_0))}{\Delta_n}R(\theta,\Delta_n, X_{i-1})+R(\theta,\Delta_n, X_{i-1}).\notag
\end{align*}

\begin{proof}[Proof of Theorem \ref{main4teo}]

We observe that
\begin{align*}
\mathbb D_{p,n}(\theta,\theta_0)=\frac1n\sum_{i=1}^n\left\{\sum_{k=1}^4(\varphi_{i,k}(\theta,\theta_0))^2+2\sum_{j<k})\varphi_{i,j}(\theta,\theta_0)\varphi_{i,k}(\theta,\theta_0)\right\}.\\
\end{align*}
Under $H_0,$ from Lemma \ref{lemmak} and Lemma \ref{lemmaer}, we derive
\begin{align*}
\frac1n\sum_{i=1}^n(\varphi_{i,1}(\theta,\theta_0))^2&=\frac1n\sum_{i=1}^n\left[\frac{(X_i-r_{k_0}(\Delta_n,X_{i-1},\theta_0))^4}{4\Delta_n^2}\left\{\frac{1}{c_{i-1}(\beta)}-\frac{1}{c_{i-1}(\beta_0)}+R(\theta,\Delta_n,X_{i-1})\right\}^2\right]\\
&=\frac1n\sum_{i=1}^n\left[\frac{(X_i-r_{k_0}(\Delta_n,X_{i-1},\theta_0))^4}{4\Delta_n^2}\left\{\frac{1}{c_{i-1}(\beta)}-\frac{1}{c_{i-1}(\beta_0)}\right\}^2\right]+\oo_{P_0}(1)\\
&\overset{P_0}{\underset{n\to\infty}{\longrightarrow}}\frac34\int  c^2(\beta_0,x)\left\{\frac{1}{c(\beta,x)}-\frac{1}{c(\beta_0,x)}\right\}^2\pi_0(\de x)\\
\frac1n\sum_{i=1}^n(\varphi_{i,2}(\theta,\theta_0))^2&=\frac1n\sum_{i=1}^n\left[\frac{(X_i-r_{k_0}(\Delta_n,X_{i-1},\theta_0))^2}{c_{i-1}^2(\beta_0)}[b_{i-1}(\alpha_0)-b_{i-1}(\alpha)]^2\right]+\oo_{P_0}(1)\\
&\overset{P_0}{\underset{n\to\infty}{\longrightarrow}}0\\
\frac1n\sum_{i=1}^n(\varphi_{i,3}(\theta,\theta_0))^2&=\frac1n\sum_{i=1}^n\left[\frac{\Delta_n^2[b_{i-1}(\alpha_0)-b_{i-1}(\alpha)]^4}{4c_{i-1}^2(\beta)}\right]+\oo_{P_0}(1)\\
&\overset{P_0}{\underset{n\to\infty}{\longrightarrow}}0\\
\frac1n\sum_{i=1}^n(\varphi_{i,4}(\theta,\theta_0))^2&=\frac1n\sum_{i=1}^n\frac14\left[\log\left(\frac{c_{i-1}(\beta)}{c_{i-1}(\beta_0)}\right)\right]^2+\oo_{P_0}(1)\\
&\overset{P_0}{\underset{n\to\infty}{\longrightarrow}}\frac14\int \left[\log\left(\frac{c(\beta,x)}{c(\beta_0,x)}\right)\right]^2\pi_0(\de x)\\
\frac1n\sum_{i=1}^n\varphi_{i,1}(\theta,\theta_0)\varphi_{i,4}(\theta,\theta_0)&=\frac1n\sum_{i=1}^n\frac{(X_i-r_{k_0}(\Delta_n,X_{i-1},\theta_0))^2}{4\Delta_n}\left\{\frac{1}{c_{i-1}(\beta)}-\frac{1}{c_{i-1}(\beta_0)}\right\}\\
&\quad\times\log\left(\frac{c_{i-1}(\beta)}{c_{i-1}(\beta_0)}\right)+\oo_{P_0}(1)\\
&\overset{P_0}{\underset{n\to\infty}{\longrightarrow}}\frac14\int c(\beta_0,x)\left\{\frac{1}{c(\beta,x)}-\frac{1}{c(\beta_0,x)}\right\}\log\left(\frac{c(\beta,x)}{c(\beta_0,x)}\right)\pi_0(\de x)\\
\frac1n\sum_{i=1}^n\varphi_{i,1}(\theta,\theta_0)\varphi_{i,j}(\theta,\theta_0)&\overset{P_0}{\underset{n\to\infty}{\longrightarrow}}0,\quad j=2,3,\\
\frac1n\sum_{i=1}^n\varphi_{i,2}(\theta,\theta_0)\varphi_{i,j}(\theta,\theta_0)&\overset{P_0}{\underset{n\to\infty}{\longrightarrow}}0,\quad j=3,4,\\
\frac1n\sum_{i=1}^n\varphi_{i,3}(\theta,\theta_0)\varphi_{i,4}(\theta,\theta_0)&\overset{P_0}{\underset{n\to\infty}{\longrightarrow}}0,
\end{align*}
uniformly in $\theta.$ Thus the statement of the theorem immediately follows.
\end{proof}

Let
\begin{align}
C_{p,n}(\theta,\theta_0):=  \left( \begin{matrix} % or pmatrix or bmatrix or Bmatrix or ...
\frac{1}{n\Delta_n}[\partial^2_{\alpha_h\alpha_k} T_{p,n}(\theta,\theta_0) ]_{\substack{h=1,...,m_1 \\k=1,...,m_1}}   &\frac{1}{n\sqrt{\Delta_n}}[	\partial^2_{\alpha_h\beta_k} T_{p,n}(\theta,\theta_0) ]_{\substack{h=1,...,m_1 \\k=1,...,m_2}}   \\
    \frac{1}{n\sqrt{\Delta_n}}[\partial^2_{\alpha_h\beta_k} T_{p,n}(\theta,\theta_0)]_{\substack{h=1,...,m_1 \\k=1,...,m_2}} &  \frac1n 	[\partial^2_{\beta_h\beta_k} T_{p,n}(\theta,\theta_0)  ]_{\substack{h=1,...,m_2 \\k=1,...,m_2}}  \\
   \end{matrix}\right)
\end{align}
where
\begin{align}
\partial^2_{\alpha_h\alpha_k} T_{p,n}(\theta,\theta_0)
&=2\sum_{i=1}^n\left\{\partial_{\alpha_h}\texttt{l}_{p,i}(\theta)\partial_{\alpha_k}\texttt{l}_{p,i}(\theta)+[\texttt{l}_{p,i}(\theta)-\texttt{l}_{p,i}(\theta_0)]\partial_{\alpha_h\alpha_k}^2\texttt{l}_{p,i}(\theta)\right\},\label{eq:firsttermC}\\
\partial^2_{\beta_h\beta_k} T_{p,n}(\theta,\theta_0)&=2\sum_{i=1}^n\left\{\partial_{\beta_h}\texttt{l}_{p,i}(\theta)\partial_{\beta_k}\texttt{l}_{p,i}(\theta)+[\texttt{l}_{p,i}(\theta)-\texttt{l}_{p,i}(\theta_0)]\partial_{\beta_h\beta_k}^2\texttt{l}_{p,i}(\theta)\right\},\label{eq:secondtermC}
\\
\partial^2_{\alpha_h\beta_k} T_{p,n}(\theta,\theta_0)&=2\sum_{i=1}^n\left\{\partial_{\alpha_h}\texttt{l}_{p,i}(\theta)\partial_{\beta_k}\texttt{l}_{p,i}(\theta)+[\texttt{l}_{p,i}(\theta)-\texttt{l}_{p,i}(\theta_0)]\partial_{\alpha_h\beta_k}^2\texttt{l}_{p,i}(\theta)\right\}.\label{}\end{align}

The following proposition concerning the asymptotic behavior of $C_{p,n}(\theta,\theta_0)$ plays a crucial role in the proof of Theorem \ref{main}.
\begin{proposition}\label{Prop}
Under $H_0,$ assume $ A_1- A_6$ and  $\Delta_n\to0,n\Delta_n\to\infty,$ as $n\to \infty$, the following convergences hold 

\begin{align}\label{eq:conC}
C_{p,n}(\theta_0,\theta_0)\overset{P_0}{\underset{n\to\infty}{\longrightarrow}} 2I(\theta_0)
\end{align}
and
\begin{align}\label{eq:conCbis}
\sup_{||\theta||\leq \varepsilon_n}||C_{p,n}(\theta_0+\theta,\theta_0)-C_{p,n}(\theta_0,\theta_0)||\overset{P_0}{\underset{n\to\infty}{\longrightarrow}}0,\quad  \varepsilon_n\to 0.
\end{align}

\end{proposition}

\begin{proof}[Proof of Proposition \ref{Prop}]
We study the uniform convergence in probability of $C_{p,n}(\theta,\theta_0).$ Thus we prove that uniformly in $\theta$
\begin{equation}\label{eq:uniflimit}
C_{p,n}(\theta,\theta_0)\overset{P_0}{\underset{n\to\infty}{\longrightarrow}} 2K(\theta,\theta_0):=
2\left(
\begin{matrix}
K_{1}(\theta,\theta_0)+K_{2}(\theta,\theta_0)& 0\\
0& K_{3}(\theta,\theta_0)+K_{4}(\theta,\theta_0)
\end{matrix}
\right)
\end{equation}
where
\begin{align*}
K_1(\theta,\theta_0)&:=\int\frac{\partial_{\alpha_h}b(\alpha,x)\partial_{\alpha_k} b(\alpha,x)}{c^2(\beta,x)}c(\beta_0,x)\pi_0(\de x),
\\
K_2(\theta,\theta_0)&:=\frac14\int \partial_{\alpha_h\alpha_k}^2\de_1(x,\theta)\left[\frac{c(\beta_0,x)}{c(\beta,x)}-1\right]\left[3\frac{c(\beta_0,x)}{c(\beta,x)}+\log\left(\frac{c(\beta,x)}{c(\beta_0,x)}\right)-1\right]\pi_0(\de x)\notag\\
&\quad+ \frac12\int\left[\frac{\partial_{\alpha_h\alpha_k}^2b(\alpha,x)(b(\alpha,x)-b(\alpha_0,x))+\partial_{\alpha_h} b(\alpha,x)\partial_{\alpha_k} b(\alpha,x)}{c(\beta,x)}\right]\\
&\quad\times\left[\frac{c(\beta_0,x)}{c(\beta,x)}-1+\log\left(\frac{c(\beta,x)}{c(\beta_0,x)}\right)\right]\pi_0(\de x)\notag\\
&\quad+\int \frac{-\partial_{\alpha_h\alpha_k}^2 b(\alpha,x)}{c(\beta,x)}\\
&\quad\times\left[\frac12\left(\frac{1}{c(\beta_0,x)}-\frac{1}{c(\beta,x)}\right)R(\theta,1,x)+\frac{c(\beta_0,x)}{c^2(\beta,x)}(b(\alpha,x)-b(\alpha_0,x))\right]\pi_0(\de x)\notag\\
%&\quad+\int\partial_{\alpha_h\alpha_k}^2\de_1(x,\theta)\frac{c(\beta,x)}{c^2(\beta_0,x)}\left[\frac34(c(\beta_0,x)-c(\beta,x))-(b(\alpha,x)-b(\alpha_0,x))\right]\pi_0(\de x)\notag\\
K_3(\theta,\theta_0)&:=\frac12\int \left\{ \frac{c(\beta_0,x)\partial_{\beta_h} c(\beta,x)\partial_{\beta_k} c(\beta,x)}{c^3(\beta,x)}\left[\frac32 \frac{c(\beta_0,x)}{c(\beta,x)}-1\right]+\frac12 	\frac{\partial_{\beta_h} c(\beta,x)\partial_{\beta_k} c(\beta,x)}{c^2(\beta,x)}\right\}\pi_0(\de x),
\\
K_4(\theta,\theta_0)&:=\frac14\int c(\beta_0,x)\partial_{\beta_h\beta_k}^2\log c(\beta,x)\left[\frac{1}{c(\beta,x)}-\frac{1}{c(\beta_0,x)}\right]\pi_0(\de x)\\
&+\frac14 \int \log\left(\frac{c(\beta,x)}{c(\beta_0,x)}\right)\frac{c(\beta_0,x)}{c(\beta,x)}\partial_{\beta_h\beta_k}^2 c^{-1}(\beta,x)\pi_0(\de x)\\
&+\frac14 \int \log\left(\frac{c(\beta,x)}{c(\beta_0,x)}\right)\partial_{\beta_h\beta_k}^2 \log c(\beta,x)\pi_0(\de x).
\end{align*}

Let us start with the analysis of the quantity $\frac{1}{n\Delta_n}\partial_{\alpha_h\alpha_k}^2 T_{p,n}(\theta,\theta_0)$ given by \eqref{eq:firsttermC} which can be split in two terms.
From \eqref{eq:deralpha} folllows that
 $$\frac{1}{n\Delta_n}\sum_{i=1}^n\partial_{\alpha_h}\texttt{l}_{p,i}(\theta)\partial_{\alpha_k}\texttt{l}_{p,i}(\theta)=\frac{1}{n\Delta_n}\sum_{i=1}^n(\eta_{i,1}^h(\theta)+\eta_{i,2}^h(\theta))(\eta_{i,1}^k(\theta)+\eta_{i,2}^k(\theta))$$ for each $\theta\in\Theta.$ Since $\partial_{\alpha_h} r_{k_0}(\Delta_n,X_{i-1},\theta)=\Delta_n\partial_{\alpha_h} b_{i-1}(\alpha)+R(\theta,\Delta_n^2, X_{i-1}),$ by taking into account Lemma \ref{lemmaer}, we get
\begin{align}\label{eta1+eta2}
 \frac{1}{n\Delta_n}\sum_{i=1}^n\partial_{\alpha_h}\texttt{l}_{p,i}(\theta)\partial_{\alpha_k}\texttt{l}_{p,i}(\theta)&=\frac{1}{n\Delta_n}\sum_{i=1}^n\eta_{i,1}^h(\theta)\eta_{i,1}^k(\theta)+\oo_{P_0}(1)\\
 &=\frac{1}{n\Delta_n}\sum_{i=1}^n\frac{\partial_{\alpha_h} b_{i-1}(\alpha)\partial_{\alpha_k} b_{i-1}(\alpha)}{c_{i-1}^2(\beta)}(X_i-r_{k_0}(\Delta_n,X_{i-1},\theta))^2+\oo_{P_0}(1)\notag\\
 &\overset{P_0}{\underset{n\to\infty}{\longrightarrow}}K_1(\theta,\theta_0)\notag
\end{align}
uniformly in $\theta.$ Now, by resorting \eqref{eq:diff} and \eqref{eq:deralphaalpha}, we rewrite the second term appearing in \eqref{eq:firsttermC} as follows
\begin{align*}
\frac{1}{n\Delta_n}\sum_{i=1}^n[\texttt{l}_{p,i}(\theta)-\texttt{l}_{p,i}(\theta_0)]\partial_{\alpha_h\alpha_k}^2\texttt{l}_{p,i}(\theta)=\frac{1}{n\Delta_n}\sum_{i=1}^n\left[\sum_{l=1}^{4}\sum_{j=1}^{4}  \varphi_{i,l}(\theta,\theta_0)\delta_{i,j}^{h,k}(\theta)\right].
\end{align*}
By applying Lemma \ref{lemma0} and Lemma \ref{lemmaer}, the following convergence results hold
\begin{align*}
&\frac{1}{n\Delta_n}\sum_{i=1}^n\varphi_{i,1}(\theta,\theta_0)\delta_{i,1}^{h,k}(\theta)\overset{P_0}{\underset{n\to\infty}{\longrightarrow}} \frac34\int \partial_{\alpha_h\alpha_k}^2\de_1(\theta,x)\frac{c^2(\beta_0,x)}{c(\beta,x)}\left[\frac{1}{c(\beta,x)}-\frac{1}{c(\beta_0,x)}\right]\pi_0(\de x),
\\
&\frac{1}{n\Delta_n}\sum_{i=1}^n\varphi_{i,1}(\theta,\theta_0)\delta_{i,2}^{h,k}(\theta)\overset{P_0}{\underset{n\to\infty}{\longrightarrow}} \frac12\int \frac{-\partial_{\alpha_h\alpha_k}^2b(\alpha,x)}{c(\beta,x)}\left[\frac{1}{c(\beta,x)}-\frac{1}{c(\beta_0,x)}\right]R(\theta,1,x)\pi_0(\de x),
\\
&\frac{1}{n\Delta_n}\sum_{i=1}^n\varphi_{i,1}(\theta,\theta_0)\delta_{i,3}^{h,k}(\theta)\overset{P_0}{\underset{n\to\infty}{\longrightarrow}} \frac14\int \partial_{\alpha_h\alpha_k}^2\ee_1(\theta,x)\left[\frac{c(\beta_0,x)}{c(\beta,x)}-1\right]\pi_0(\de x),
\\
&\frac{1}{n\Delta_n}\sum_{i=1}^n\varphi_{i,1}(\theta,\theta_0)\delta_{i,4}^{h,k}(\theta)\\
&\overset{P_0}{\underset{n\to\infty}{\longrightarrow}} \frac12\int\left[\frac{c(\beta_0,x)}{c(\beta,x)}-1\right]\left[\frac{\partial_{\alpha_h\alpha_k}^2b(\alpha,x)(b(\alpha,x)-b(\alpha_0,x))+\partial_{\alpha_h} b(\alpha,x)\partial_{\alpha_k} b(\alpha,x)}{c(\beta,x)}\right]\pi_0(\de x),
\\
&\frac{1}{n\Delta_n}\sum_{i=1}^n\varphi_{i,2}(\theta,\theta_0)\delta_{i,2}^{h,k}(\theta)\overset{P_0}{\underset{n\to\infty}{\longrightarrow}} \int \frac{c(\beta_0,x)}{c^2(\beta,x)}(-\partial_{\alpha_h\alpha_k}^2 b(\alpha,x))(b(\alpha,x)-b(\alpha_0,x))\pi_0(\de x),
\\
%&\frac{1}{n\Delta_n}\sum_{i=1}^n\varphi_{i,2}(\theta,\theta_0)\delta_{i,3}^{h,k}(\theta)\overset{P_0}{\underset{n\to\infty}{\longrightarrow}}  \int \frac{c(\beta_0,x)}{c^2(\beta,x)}\partial_{\alpha_h\alpha_k}^2 \ee_1(\theta,x))(b(\alpha,x)-b(\alpha_0,x))\pi_0(\de x),\\
&\frac{1}{n\Delta_n}\sum_{i=1}^n\varphi_{i,4}(\theta,\theta_0)\delta_{i,1}^{h,k}(\theta)\overset{P_0}{\underset{n\to\infty}{\longrightarrow}}\frac14 \int \log\left(\frac{c(\beta,x)}{c(\beta_0,x)}\right)\frac{c(\beta_0,x)}{c(\beta,x)}\partial_{\alpha_h\alpha_k}^2 \de_1(\theta,x)\pi_0(\de x),\\
&\frac{1}{n\Delta_n}\sum_{i=1}^n\varphi_{i,4}(\theta,\theta_0)\delta_{i,3}^{h,k}(\theta)\overset{P_0}{\underset{n\to\infty}{\longrightarrow}}\frac14\int  \partial_{\alpha_h\alpha_k}^2\ee_1(\theta,x)\log\left(\frac{c(\beta,x)}{c(\beta_0,x)}\right)\pi_0(\de x),
\\
&\frac{1}{n\Delta_n}\sum_{i=1}^n\varphi_{i,4}(\theta,\theta_0)\delta_{i,4}^{h,k}(\theta)\\&\overset{P_0}{\underset{n\to\infty}{\longrightarrow}}\frac12\int\log\left(\frac{c(\beta,x)}{c(\beta_0,x)}\right)\left\{\frac{\partial_{\alpha_h\alpha_k}^2 b(\alpha,x)(b(\alpha,x)-b(\alpha_0,x))+\partial_{\alpha_k} b(\alpha,x)\partial_{\alpha_h} b(\alpha,x)}{c(\beta,x)}\right\}\pi_0(\de x),\\
&\frac{1}{n\Delta_n}\sum_{i=1}^n\varphi_{i,2}(\theta,\theta_0)\delta_{i,j}^{h,k}(\theta)\overset{P_0}{\underset{n\to\infty}{\longrightarrow}} 0,\quad j=1,3,4,
\\
&\frac{1}{n\Delta_n}\sum_{i=1}^n\varphi_{i,3}(\theta,\theta_0)\delta_{i,j}^{h,k}(\theta)\overset{P_0}{\underset{n\to\infty}{\longrightarrow}} 0,\quad j=1,2,3,4,\\
&\frac{1}{n\Delta_n}\sum_{i=1}^n\varphi_{i,4}(\theta,\theta_0)\delta_{i,2}^{h,k}(\theta)\overset{P_0}{\underset{n\to\infty}{\longrightarrow}} 0,
\end{align*}
uniformly in $\theta.$
Finally, since $\de_1(\theta,x)=-\ee_1(\theta,x),$ we get
\begin{align}\label{eq:firsttermCconv}
\frac{1}{n\Delta_n}\sum_{i=1}^n[\texttt{l}_{p,i}(\theta)-\texttt{l}_{p,i}(\theta_0)]\partial_{\alpha_h\alpha_k}^2\texttt{l}_{p,i}(\theta)
\overset{P_0}{\underset{n\to\infty}{\longrightarrow}} K_2(\theta,\theta_0).
\end{align}
uniformly in $\theta.$ Hence, by \eqref{eta1+eta2} and \eqref{eq:firsttermCconv}, we immediately derive

\begin{align}\label{eq:conK1+K2}
\frac{1}{n\Delta_n}\partial_{\alpha_h\alpha_k}^2 T_{p,n}(\theta,\theta_0)
\overset{P_0}{\underset{n\to\infty}{\longrightarrow}} 2(K_1(\theta,\theta_0)+K_2(\theta,\theta_0))
\end{align}
uniformly in $\theta.$

Now, we consider the elements of the matrix $C_{n,p}(\theta,\theta_0)$ given by \eqref{eq:secondtermC}.
First, we study the convergence probability of  $$\frac{1}{n}\sum_{i=1}^n \partial_{\beta_h}  \texttt{l}_{p,i}(\theta) \partial_{\beta_k}  \texttt{l}_{p,i}(\theta)= \frac{1}{n}\sum_{i=1}^n (\xi_{i,1}^h(\theta)+\xi_{i,2}^h(\theta)+\xi_{i,3}^h(\theta))(\xi_{i,1}^k(\theta)+\xi_{i,2}^k(\theta)+\xi_{i,3}^k(\theta)).$$
Since $\partial_{\beta_h} r_{k_0}(\Delta_n,X_{i-1},\theta)=R(\theta,\Delta_n^2, X_{i-1}),$ from Lemma \ref{lemmaer} and Lemma \ref{lemma0} we derive
\begin{align}\label{eq:conK3}
\frac{1}{n}\sum_{i=1}^n \partial_{\beta_h}  \texttt{l}_{p,i}(\theta) \partial_{\beta_k}  \texttt{l}_{p,i}(\theta)&=\frac{1}{n}\sum_{i=1}^n \xi_{i,2}^h(\theta)\xi_{i,2}^k(\theta)+\oo_{P_0}(1)\\
&=\frac{1}{n}\sum_{i=1}^n\frac{\partial_{\beta_h} c_{i-1}(\beta)\partial_{\beta_k} c_{i-1}(\beta)}{4\Delta_n^2c_{i-1}^4(\beta)}(X_i-r_{k_0}(\Delta_n,X_{i-1},\theta))^4\notag\\
&\quad+\frac{1}{n}\sum_{i=1}^n\frac{\partial_{\beta_h} c_{i-1}(\beta)\partial_{\beta_k} c_{i-1}(\beta)}{2\Delta_nc_{i-1}^3(\beta)}(X_i-r_{k_0}(\Delta_n,X_{i-1},\theta))^2\notag\\
&\quad+\frac{1}{n}\sum_{i=1}^n\frac{\partial_{\beta_h} c_{i-1}(\beta)\partial_{\beta_k} c_{i-1}(\beta)}{4c_{i-1}^2(\beta)}+\oo_{P_0}(1)\notag\\
&\overset{P_0}{\underset{n\to\infty}{\longrightarrow}}K_3(\theta,\theta_0)\notag
\end{align}
uniformly in $\theta.$ Now, by resorting \eqref{eq:diff} and \eqref{eq:derbeta}, we rewrite the second term appearing in \eqref{eq:secondtermC} as follows

\begin{align*}
\frac{1}{n}\sum_{i=1}^n[\texttt{l}_{p,i}(\theta)-\texttt{l}_{p,i}(\theta_0)]\partial_{\beta_h\beta_k}^2\texttt{l}_{p,i}(\theta)=\frac{1}{n}\sum_{i=1}^n\left[\sum_{k=1}^{4}\sum_{j=1}^{3} \varphi_{i,k}(\theta,\theta_0)\nu_{i,j}^{h,k}(\theta)\right].
\end{align*}

By taking into account again Lemma \ref{lemma0} and Lemma \ref{lemmaer}, the following results yield
\begin{align*}
&\frac{1}{n}\sum_{i=1}^n \varphi_{i,1}(\theta,\theta_0)\nu_{i,3}(\theta)
\overset{P_0}{\underset{n\to\infty}{\longrightarrow}} \frac14\int c(\beta_0,x)\partial_{\beta_h\beta_k}^2\log c(\beta,x)\left[\frac{1}{c(\beta,x)}-\frac{1}{c(\beta_0,x)}\right]\pi_0(\de x)
\\
&\frac{1}{n}\sum_{i=1}^n \varphi_{i,4}(\theta,\theta_0)\nu_{i,1}(\theta)
\overset{P_0}{\underset{n\to\infty}{\longrightarrow}}\frac14 \int \log\left(\frac{c(\beta,x)}{c(\beta_0,x)}\right)\frac{c(\beta_0,x)}{c(\beta,x)}\partial_{\beta_h\beta_k}^2 c^{-1}(\beta,x)\pi_0(\de x)
\\
&\frac{1}{n}\sum_{i=1}^n \varphi_{i,4}(\theta,\theta_0)\nu_{i,3}(\theta)
\overset{P_0}{\underset{n\to\infty}{\longrightarrow}}\frac14 \int \log\left(\frac{c(\beta,x)}{c(\beta_0,x)}\right)\partial_{\beta_h\beta_k}^2 \log c(\beta,x)\pi_0(\de x)\\
&\frac{1}{n}\sum_{i=1}^n \varphi_{i,1}(\theta,\theta_0)\nu_{i,j}(\theta)
\overset{P_0}{\underset{n\to\infty}{\longrightarrow}}0,\quad j=1,2,\\
&\frac{1}{n}\sum_{i=1}^n \varphi_{i,k}(\theta,\theta_0)\nu_{i,j}(\theta)
\overset{P_0}{\underset{n\to\infty}{\longrightarrow}}0,\quad k,j=1,2,3,\\
&\frac{1}{n}\sum_{i=1}^n \varphi_{i,4}(\theta,\theta_0)\nu_{i,2}(\theta)
\overset{P_0}{\underset{n\to\infty}{\longrightarrow}}0,
\end{align*}
uniformly in $\theta.$
Finally
\begin{align}\label{eq:conK4}
\frac{1}{n}\sum_{i=1}^n[\texttt{l}_{p,i}(\theta)-\texttt{l}_{p,i}(\theta_0)]\partial_{\beta_h\beta_k}^2\texttt{l}_{p,i}(\theta)\overset{P_0}{\underset{n\to\infty}{\longrightarrow}}K_4(\theta,\theta_0)\end{align}
uniformly in $\theta.$ Therefore, by \eqref{eq:conK3} and  \eqref{eq:conK4}, we get 
\begin{align}\label{eq:conK3+K4}
\frac1n\partial_{\beta_h\beta_k}^2 T_{p,n}(\theta,\theta_0)
\overset{P_0}{\underset{n\to\infty}{\longrightarrow}} 2(K_3(\theta,\theta_0)+K_4(\theta,\theta_0))
\end{align}
uniformly in $\theta.$

Recalling the expressions \eqref{eq:deralpha}, \eqref{eq:derbeta}, \eqref{eq:dersecalphabeta} and \eqref{eq:diff}, by means of similar arguments adopted above, it is not hard to prove that
\begin{equation*}
 \frac{1}{n\sqrt{\Delta_n}}\sum_{i=1}^n \partial_{\alpha_h} \texttt{l}_{p,i}(\theta)\partial_{\beta_k}  \texttt{l}_{p,i}(\theta) \overset{P_0}{\underset{n\to\infty}{\longrightarrow}}   0
 \end{equation*}
and
 \begin{equation*}
 \frac{1}{n\sqrt{\Delta_n}}\sum_{i=1}^n [\texttt{l}_{p,i}(\theta)-\texttt{l}_{p,i}(\theta_0)]\partial_{\alpha_h\beta_k}^2  \texttt{l}_{p,i}(\theta) \overset{P_0}{\underset{n\to\infty}{\longrightarrow}}   0
 \end{equation*}
 uniformly in $\theta.$ This implies that
 \begin{align}\label{eq:con0}
\frac{1}{n\sqrt{\Delta_n}}\partial_{\alpha_h\beta_k}^2 T_{p,n}(\theta,\theta_0)
\overset{P_0}{\underset{n\to\infty}{\longrightarrow}} 0
\end{align}
 uniformly in $\theta.$  
 
 In conclusion the results \eqref{eq:conK1+K2}, \eqref{eq:conK3+K4} and \eqref{eq:con0} lead to the convergence \eqref{eq:uniflimit}. Moreover, immediately  \eqref{eq:uniflimit} implies \eqref{eq:conC} since $K(\theta_0,\theta_0)=I(\theta_0)$. From the inequality
 
\begin{align*}
&\sup_{||\theta||\leq \varepsilon_n}||C_{p,n}(\theta_0+\theta,\theta_0)-C_{p,n}(\theta_0,\theta_0)||\\
&\leq \sup_{||\theta||\leq \varepsilon_n}||C_{p,n}(\theta_0+\theta,\theta_0)-2K(\theta_0+\theta,\theta_0)||+\sup_{||\theta||\leq \varepsilon_n}||2K(\theta_0+\theta,\theta_0)-2I(\theta_0)||\\
&\quad+||2I(\theta_0)-C_{p,n}(\theta_0,\theta_0)||
\end{align*}
follows \eqref{eq:conCbis}. Indeed,  \eqref{eq:conC} leads to $||2I(\theta_0)-C_{p,n}(\theta_0,\theta_0)||{\underset{n\to\infty}{\longrightarrow}} 0,\varepsilon_n\to0,$ while the term $\sup_{||\theta||\leq \varepsilon_n}||C_{p,n}(\theta_0+\theta,\theta_0)-2K(\theta_0+\theta,\theta_0)||\overset{P_0}{\underset{n\to\infty}{\longrightarrow}} 0,\varepsilon_n\to0,$ by the uniformity of the convergence (i.e. by the result \eqref{eq:uniflimit}). Furthermore, $\sup_{||\theta||\leq \varepsilon_n}||K(\theta_0+\theta,\theta_0)-I(\theta_0)||\overset{P_0}{\underset{n\to\infty}{\longrightarrow}} 0,\varepsilon_n\to0,$ because the assumptions $A_3$ and $A_5,$ imply that $K(\theta,\theta_0)$ is a continuous function with respect to $\theta.$ \end{proof} 

Now, we are able to prove Theorem \ref{main}.

\begin{proof}[Proof of Theorem \ref{main}.]

We adopt classical arguments. By Taylor's formula, we have that
\begin{align}\label{eq:taylor}
T_{p,n}(\hat\theta_{p,n},\theta_0)
&=T_{p,n}(\theta_0,\theta_0)+n\partial_\theta T_{p,n}(\theta_0,\theta_0)(\hat\theta_{p,n}-\theta_0)\\
&\quad+\frac12(\varphi(n)^{-1/2}(\hat\theta_n-\theta_0))'\Lambda_{p,n}(\hat\theta_{p,n},\theta_0))\varphi(n)^{-1/2}(\hat\theta_{p,n}-\theta_0)\notag\\
&=\frac12(\varphi(n)^{-1/2}(\hat\theta_n-\theta_0))'\Lambda_{p,n}(\hat\theta_{p,n},\theta_0)\varphi(n)^{-1/2}(\hat\theta_n-\theta_0)\notag
\end{align}
where in the last step we denoted by
\begin{align*}
\Lambda_{p,n}(\hat\theta_{p,n},\theta_0)&:=\varphi(n)^{1/2}\int_0^1(1-u)\partial_\theta^2 T_{p,n}(\theta_0+u(\hat\theta_{p,n}-\theta_0),\theta_0)\de u\varphi(n)^{1/2}\\
&=\int_0^1(1-u)[C_{p,n}(\theta_0+u(\hat\theta_{p,n}-\theta_0),\theta_0)-C_{p,n}(\theta_0,\theta_0)]\de u+\frac12C_{p,n}(\theta_0,\theta_0).
\end{align*}
Proposition \ref{Prop} implies  \begin{equation}\label{eq:conmainteo}
\Lambda_{p,n}(\hat\theta_{p,n},\theta_0)\overset{P_0}{\underset{n\to\infty}{\longrightarrow}}2I(\theta_0).
\end{equation} 
By taking into account \eqref{eq:taylor}, \eqref{eq:conest} and \eqref{eq:conmainteo}, Slutsky's theorem allows to conclude the proof.
\end{proof}

\begin{proof}[Proof of Theorem \ref{main3teo}]
Under $H_{1,n}$ we have that (see Lemma 2 in \cite{kituch})
$$\varphi(n)^{-1/2}(\hat \theta_{p,n}-(\theta_0+\varphi(n)h))\overset{d}{\underset{n\to\infty}{\longrightarrow}} N(0,I(\theta_0)^{-1}).$$
Therefore, under the hypothesis $H_{1,n}$
$$\varphi(n)^{-1/2}(\hat \theta_{p,n}-\theta_0)=\varphi(n)^{-1/2}(\hat \theta_{p,n}-\theta)+h\overset{d}{\underset{n\to\infty}{\longrightarrow}} N(h,I(\theta_0)^{-1})$$
and
$$C_{p,n}(\hat\theta_{p,n},\theta_0)\overset{P_\theta}{\underset{n\to\infty}{\longrightarrow}}2I(\theta_0) \quad (\text{under}\, H_{1,n}).$$
Hence, from \eqref{eq:taylor} we obtain the result \eqref{main3}.
\end{proof}

%\begin{lemma}\label{lemmaku}
%Let $X_n\overset{P}{\underset{n\to\infty}{\longrightarrow}}c>0$. Then for any $\varepsilon >0,$ as $n\to\infty,$
%$$P\left(X_n\leq \frac{\varepsilon}{n} \right)\to 0.$$
%\end{lemma}
%\begin{proof}
%See the proof of Lemma 3 in \cite{kituch}.
%\end{proof}

\section*{Acknowlegments}
We would like to thank both the referees for their comments which have greatly improved the first version of the manuscript. 

\section*{Conflict of interest}

On behalf of all authors, the corresponding author states that there is no conflict of interest.

\end{document}